\documentclass[leqno,11pt]{amsart}

\usepackage{amssymb, amsmath}
\usepackage{amssymb,amsfonts}
\usepackage{amsmath,latexsym}
\usepackage{pdfsync}
\usepackage{bbm}

\newtheorem{thm}{Theorem}[section]
\newtheorem{prop}{Proposition}[section]
\newtheorem{lem}{Lemma}[section]
\newtheorem{defi}{Definition}[section]
\newtheorem{col}{Corollary}[section]

\newtheorem {rmk} {Remark} [section]

\numberwithin{equation}{section}

\makeatletter
\def\sommaire{\@restonecolfalse\if@twocolumn\@restonecoltrue\onecolumn
\fi\chapter*{Sommaire\@mkboth{SOMMAIRE}{SOMMAIRE}}
  \@starttoc{toc}\if@restonecol\twocolumn\fi}
\makeatother

\makeatletter
\def\thebibliographie#1{\chapter*{Bibliographie\@mkboth
  {BIBLIOGRAPHIE}{BIBLIOGRAPHIE}}\list
  {[\arabic{enumi}]}{\settowidth\labelwidth{[#1]}\leftmargin\labelwidth
  \advance\leftmargin\labelsep
  \usecounter{enumi}}
  \def\newblock{\hskip .11em plus .33em minus .07em}
  \sloppy\clubpenalty4000\widowpenalty4000
  \sfcode`\.=1000\relax}

\makeatother

\makeatletter
\def\references#1{\section*{R\'ef\'erences\@mkboth
  {R\'EF\'ERENCES}{R\'EF\'ERENCES}}\list
  {[\arabic{enumi}]}{\settowidth\labelwidth{[#1]}\leftmargin\labelwidth
  \advance\leftmargin\labelsep
  \usecounter{enumi}}
  \def\newblock{\hskip .11em plus .33em minus .07em}
  \sloppy\clubpenalty4000\widowpenalty4000
  \sfcode`\.=1000\relax}

\makeatother

\def\refer#1{~\ref{#1}}

\def\inte#1{
\displaystyle\mathop{#1\kern0pt}^\circ
}

\newcommand{\beq}{\begin{equation}}
\newcommand{\eeq}{\end{equation}}
\newcommand{\ben}{\begin{eqnarray}}
\newcommand{\een}{\end{eqnarray}}
\newcommand{\beno}{\begin{eqnarray*}}
\newcommand{\eeno}{\end{eqnarray*}}
\def\hm{{H^{-\frac12,0}}}

\let\d=\delta
\let\e=\varepsilon

\let\lam=\lambda

\let\f=\frac

\let\p=\partial
\let\om=\omega

\let\D=\Delta

\let\Om=\Omega
\let\wt=\widetilde
\let\wh=\widehat

\def\dB{\dot{B}}
\def\dH{\dot{H}}
\def\th{\theta}

\def\cA{{\mathcal A}}
\def\cB{{\mathcal B}}
\def\cC{{\mathcal C}}

\def\cF{{\mathcal F}}

\def\cS{{\mathcal S}}

\def\virgp{\raise 2pt\hbox{,}}
\def\cdotpv{\raise 2pt\hbox{;}}

\def\eqdefa{\buildrel\hbox{\footnotesize def}\over =}

\def\C{\mathop{\mathbb C\kern 0pt}\nolimits}
\def\DD{\mathop{\mathbb D\kern 0pt}\nolimits}
\def\EE{\mathop{\mathbb E\kern 0pt}\nolimits}
\def\K{\mathop{\mathbb K\kern 0pt}\nolimits}
\def\N{\mathop{\mathbb  N\kern 0pt}\nolimits}
\def\Q{\mathop{\mathbb  Q\kern 0pt}\nolimits}
\def\R{{\mathop{\mathbb R\kern 0pt}\nolimits}}
\def\SS{\mathop{\mathbb  S\kern 0pt}\nolimits}
\def\St{\mathop{\mathbb  S\kern 0pt}\nolimits}
\def\Z{\mathop{\mathbb  Z\kern 0pt}\nolimits}
\def\ZZ{{\mathop{\mathbb  Z\kern 0pt}\nolimits}}
\def\H{{\mathop{{\mathbb  H\kern 0pt}}\nolimits}}
\def\PP{\mathop{\mathbb P\kern 0pt}\nolimits}
\def\TT{\mathop{\mathbb T\kern 0pt}\nolimits}

\def\h {{\rm h}}
\def\v {{\rm v}}
\def\na{\nabla}

\def\buh{\bar{u}^{\rm h}}

\newcommand{\la}{\lambda}

\newcommand{\andf}{\quad\hbox{and}\quad}
\newcommand{\with}{\quad\hbox{with}\quad}

\def\omss{\omega_{\frac34}}

\def\dive{\mathop{\rm div}\nolimits}
\def\curl{\mathop{\rm curl}\nolimits}

\def\Supp{\mathop{\rm Supp}\nolimits\ }

\begin{document}

\title[Navier-Stokes system with
strong   dissipation in one direction]
{Global strong solutions to 3-D Navier-Stokes system with
strong dissipation in one direction}
\author[M. PAICU]{Marius Paicu}
\address [M. PAICU]%
{Universit\'e  Bordeaux 1\\
 Institut de Math\'ematiques de Bordeaux\\
F-33405 Talence Cedex, France} \email{marius.paicu@math.u-bordeaux1.fr}
\author[P. ZHANG]{Ping Zhang}%
\address[P. ZHANG]
 {Academy of Mathematics $\&$ Systems Science
and  Hua Loo-Keng Key Laboratory of Mathematics, The Chinese Academy of
Sciences, Beijing 100190, CHINA, and School of Mathematical Sciences, University of Chinese Academy of Sciences, Beijing 100049, CHINA. } \email{zp@amss.ac.cn}
\date{\today}
\maketitle

\begin{abstract}

We consider three dimensional incompressible Navier-Stokes equation $(NS)$ with different viscous coefficient in the vertical and horizontal variables. In particular,
 when one of these viscous coefficients is large enough compared to the initial data, we prove the global well-posedness of this system.
 In fact, we obtain the existence of a global strong solution to $(NS)$ when the initial data verify an anisotropic
  smallness condition which takes into account the different roles of the horizontal and vertical viscosity.

\end{abstract}

\setcounter{equation}{0}
\section{Introduction}

In this paper, we consider 3-D anisotropic incompressible Navier-Stokes system with different vertical  and horizontal viscosity.
Our goal is to  study the role of a large viscous coefficient  in one direction  that plays
in obtaining the global existence of  strong solutions to Navier-Stokes system.   Let us recall  Navier-Stokes equations which describes
  the evolution of viscous fluid in  $\R^3:$
\begin{equation}
\label{NS}
\begin{cases}
\partial_t u+u\cdot\nabla u-\nu_\h\Delta_\h u-\nu_{\v}\partial_3^2 u=-\nabla p\hskip0.3cm\text{in}\hskip0.3cm[0,T]\times \R^3,\\
\dive u=0,\\
u|_{t=0}=u_0,
\end{cases}
\end{equation}
where $\Delta_\h\eqdefa \partial_{x_1}^2+\partial_{x_2}^2$ designates the horizontal Laplacian, $\nu_\h>0$ and $\nu_\v>0$ are respectively the horizontal and the vertical viscous coefficients, $u=(u_1,u_2,u_3)$  denotes  the fluid velocity, and $p$ the scalar pressure function, which guarantees the divergence free condition of the velocity field.

 \smallskip

When $\nu_\h=\nu_\v=\nu>0,$ \eqref{NS} is exactly the classical Navier-Stokes system. In the sequel, we shall always denote this system  by $(NS_\nu).$
 Whereas when $\nu_\h=\nu>0$ and $\nu_\v=0,$ $(\ref{NS})$ reduces to the
anisotropic Navier-Stokes system arising from geophysical fluid mechanics (see \cite{CDGG}). We remark that  Navier-Stokes system with large vertical viscosity
 is a usual model to study the evolution of the fluid in a thin domain in the vertical direction (see (2.4) of  \cite{raugelsell} for instance).

\smallskip

We begin by recalling some important and classical  results about Navier-Stokes system. Especially, we shall
focus on the conditions which guarantee  the global existence of  strong solution to $(NS_\nu).$

The first important result about the classical Navier-Stokes system was obtained by Leray  in the seminar paper \cite{leray} in  1933. He proved
that given an arbitrary  finite energy solenoidal vector field,  $(NS_\nu)$  has  a global in time weak
solution which verifies the energy inequality. This solution is unique and regular for positive time in $\R^2,
$ but unfortunately the uniqueness and regularities of such solution in three space dimension is still one of most challenging open questions in the field of
 mathematical fluid mechanics.
 Fujita-Kato  \cite{fujitakato} gave a partial answer
to the construction of global unique solution to  $(NS_\nu)$. Indeed, the theorem of Fujita-Kato  \cite{fujitakato} allows to construct local in time unique solution
 to \eqref{NS} with initial data in the homogeneous Sobolev spaces  $\dot H^{\frac 12}(\R^3)$, or in the Lebsegue space $L^3(\R^3)$ (see \cite{kato}). Moreover,  if the initial data is small enough compared to the the viscosity, that is,  $\|u_0\|_{\dot H^{\frac 12}}\leq c\nu$,  for some sufficiently small constant $c,$ then the strong solution exists globally in time.

The above  result was extended by  Cannone, Meyer and Planchon \cite{cannonemeyerplanchon} for initial data in Besov spaces with negative index.
More precisely, they proved that if the initial data belongs to the Besov space,
$\dot B^{-1+\frac 3p}_{p,\infty}(\R^3)$ for  some $p\in ]3,\infty[$ and its norm is sufficiently small  compared to the viscosity, then $(NS_\nu)$ has a unique global solution.
The typical example of such kind of initial data reads
$$u_0^\epsilon=\epsilon^{-\alpha}\sin(\frac{x_3}{\epsilon})(-\partial_2\varphi(x),\partial_1\varphi(x),0),$$
where $0<\alpha<1$ and $\varphi\in\mathcal S(\R^3;\R)$. We remark that this type of initial data is not small in either $\dot H^{\frac 12}(\R^3)$  or $L^3(\R^3).$

 The end-point result in this direction is given by Koch and  Tataru \cite{kochtataru}. They proved that given initial data
 in the derivatives of $\text{BMO}$ space and its norm is sufficiently small compared to the viscosity, then  $(NS_\nu)$  has a unique global solution (one may check
 \cite{ZZ12} for the regularities of such solutions).
We remark that for $p\in ]3,\infty[$, there holds
$$\dot H^{\f12}(\R^3) \hookrightarrow L^3(\R^3) \hookrightarrow \dot B^{-1+\frac 3p}_{p,\infty}(\R^3) \hookrightarrow \text{BMO}^{-1}(\R^3)
 \hookrightarrow \dot B^{-1}_{\infty,\infty}(\R^3),$$
and the norms to the above spaces are sclaing-invariant
under the following scaling transformation
\beq \label{S1eq2} u_\lambda(t,x)\eqdefa\lambda u(\lambda^2 t,\lambda x) \andf u_{0,\lambda}(x)\eqdefa \lambda u_0(\lambda x).\eeq
We notice that for any solution $u$ of $(NS_\nu)$ on $[0,T],$ $u_\lam$ determined by \eqref{S1eq2} is also a solution of $(NS_\nu)$ on $[0,T/\lam^2].$ We point out that the largest space, which belongs to $\cS'(\R^3)$ and the norm of which is scaling invariant under \eqref{S1eq2}, is $\dB^{-1}_{\infty,\infty}(\R^3)$.
Moreover, Bourgain and Pavlovi\'c \cite{BP08} proved that $(NS_\nu)$ is actually
ill-posed with initial data in $\dot B^{-1}_{\infty,\infty}(\R^3).$

We mention  some examples of large initial data which generate unique global
solution to $(NS_\nu)$. First of all,
 Raugel and Sell \cite{raugelsell} obtained the global well-posedness of $(NS_\nu)$ in a thin enough domain. This result was extended
by Gallagher \cite{G97} that $(NS_\nu)$ has a unique global  periodic solution provided that
the initial data $u_0$ can be split as $u_0=v_0+w_0$, where $v_0$ is a bi-dimensional solenoidal vector field in $L^2({\Bbb{T}}^2_\h)$
and $w_0\in H^{\frac 12}(\Bbb{T}^3)$ which satisfy
$$\|w_0\|_{H^{\frac 12}(\Bbb{T}^3)}\exp \bigg(\frac{\|u_0\|_{L^2(\Bbb{T}^2_\h)}^2}{\nu^2}\bigg)\leq c\nu$$
for some $c$ being sufficiently small.

  In 2006, Chemin and Gallagher \cite{cgens} constructed the following example of initial
data which  generates a unique global solution to $(NS_\nu)$, and which is large in $\dot B^{-1}_{\infty,\infty}(\R^3)$ and is strongly oscillatory in one
direction,
$$u_0^N=\bigl(Nu_\h(x_h)\cos(Nx_3),-\dive_\h u_h(x_h)sin(Nx_3)\bigr),$$
where $\|u_\h\|_{L^2(\Bbb{T}^2_h)}\leq C(\ln N)^{\frac 19}$.

The other  interesting class of large initial data
which generate unique global  solutions to $(NS_\nu)$
is the so-called
slowly varying data,
\begin{equation*}\label{slowinitialdata}
u_0^\varepsilon(x)=(v^\h_0 +\varepsilon w^\h_0,w^3_0)(x_\h,\varepsilon x_3)
\end{equation*}
which was introduced by Chemin and Gallagher
in \cite{CG10} (see also \cite{CGZ, CZ5}).

On the other hand, by crucially using the fact that $\dive u=0,$ Zhang \cite{Zhang10} and the authors
\cite{PZ1} improved  Fujita-Kato's result
by requiring only two components of the initial velocity being
sufficiently small in some critical Besov space
even when $\nu_\h>0$ and $\nu_\v=0$ in (\ref{NS}).

Liu and the second author \cite{LZ2} first investigated the global well-podesness of \eqref{NS} with $\nu_\h=1$ and $\nu_\v$ being
sufficiently large. In particular, they proved the following result:

\begin{thm}\label{thmLZ}
{\sl Let $\nu_\h=1,$ $u_0$ satisfy $\curl u_0\in L^{\frac 32}(\R^3)$ and $\dive u_0=0.$
Then there exists some universal positive constant $C_1$ such that if
\begin{equation} \label{thmmaincondi}
\nu_\v\geq C_1\bigl(M_0+M_0^{\frac14}\bigr) \with M_0\eqdefa \|\omega_0\|_{L^{\f32}}^{\f32}+\|\nabla u^3_0\|_{\hm}^2 \andf \om_0\eqdefa\p_1u_0^2-\p_2u_0^1,
\end{equation}
(\ref{NS}) has a unique global solution $u\in C\bigl([0,\infty[,H^{\f12}(\R^3)\bigr)
\bigcap L^2_{\rm{loc}}\bigl(\R^+; H^{\frac32}(\R^3)\bigr)$
so that for any $t>0,$
\begin{equation}\label{thmmainestimate}
\bigl\|\om(t)\bigr\|_{L^{\f32}}^{\f32}+\|\nabla u^3(t)\|_{\hm}^2
+\int_0^t\left(\bigl\|\nabla_{\nu_\v}\omss\bigr\|_{L^2}^2
+\bigl\|\nabla_{\nu_\v} \nabla v^3\bigr\|_{\hm}^2\right)\, dt'
\leq 2M_0,
\end{equation}
 where $\nabla_{\nu_\v}\eqdefa \left(\na_\h, \nu_\v\p_3\right).$
}\end{thm}

Our goal in this paper is to study the role of one big viscosity that plays in
obtaining global existence of strong solutions to \eqref{NS} even the  viscous coefficients in other variables
are small. One of the consequence of our
result below (see Theorem \ref{thPZ9}) ensures that given regular initial data $u_0,$  if  the horizontal viscosity in \eqref{NS} is small like
$\nu_h=\epsilon$ and the vertical one is big enough like
$\frac{\frak{K}}{\epsilon^{15}}$ for some large enough constant $\frak{K},$ then \eqref{NS}  has a unique global solution.
This case is obviously not covered by Theorem \ref{thmLZ}. Furthermore, here we allow the initial data to have a large and slowly varying part.

The  main result of this paper states as follows:

\begin{thm}\label{thPZ9}
 {\sl Let~ $v_0\in\cB^{\f12,0}\cap B^{0,\f12}(\R^3)$ be a solenoidal vector field and $\buh_0$ be a horizontal solenoidal vector field in~$H^1(\R^3)\cap
L^\infty(\R_\v;\dot H^{-\delta} (\R_\h))\cap L^\infty(\R_\v;H^1(\R_\h))$ for
some $\d\in]0,1[.$ We  also assume  that ~$\buh_0,$ ~$\partial_z \buh_0$
and~$\partial_z^2 \buh_0$ belongs to~$L^2(\R_\v;\dot H^{-\frac
12}\cap\dot H^{\frac 12} (\R_\h)).$ 
Then  for any $\th\in [0,1/2[,$ there exists a positive constant $\cA_{\d,\nu_\h}(\buh_0),$ which
depends on the norms of $\buh_0$ above, such that if  $\nu_\v\geq \nu_\h>0$ and  $\nu_\v$ is so large that \beq \label{S1eq1}
\begin{split}
\Bigl(\cA_0^{\f12}\bigl(1+\cA_0^{\f12}\bigr)\bigl(\nu_\h^{-\f78}+\nu_\h^{\f18}\bigr)\nu_\v^{-\f18}+&\e^{1-\th}\nu_\h^{-1+\f{\th}2}\nu_\v^{-\f{\th}2}\Bigr)\cA_{\d,\nu_h}(\buh_0)\leq c_0\nu_\h \with \\ &\qquad\cA_0\eqdefa\|v_0\|_{\cB^{\f12,0}}\bigl(\|v_0\|_{B^{0,\f12}}+\|v_0\|_{\cB^{\f12,0}}\bigr)
\end{split}
\eeq
for some $c_0$ sufficiently small,
 the initial data
\beq\label{S1eq3}
u_{0,\e}(x_\h, x_3) =\bigl( \buh_0(x_\h, \e x_3),0\bigr)+v_0(x)
\eeq
generates a unique global solution $u$ to \eqref{NS} in the space~$C_b(\R^+;
B^{0,\frac 12} (\R^3))$ with $\na u\in L^2(\R^+; B^{0,\frac 12}(\R^3))$. }
\end{thm}

The definitions of the Besov norms will be presented in Section \ref{Sect3}.
The exact form of the constant $\cA_{\d,\nu_h}(\buh_0)$ will be given by \eqref{S5eq4}.

\begin{rmk}
\begin{itemize}

\item[(1)] We remark that the norms of $v_0$ in both $\cB^{\f12,0}(\R^3)$ and $ B^{0,\f12}(\R^3)$ are scaling invariant under the scaling
transformation \eqref{S1eq2}.

\item[(2)] Similar result holds for \eqref{NS} in a regular bounded domain with Dirichlet boundary condition for the velocity field. Indeed the
first eigenvalue of  the operator,
$A\eqdefa -\nu_\h\D_\h-\nu_\v\partial_3^2$,  is bigger than $\max\left(\nu_\h,\nu_\v\right)$. On the other hand,  Avrin \cite{Avrin} proved that as long as
$\|u_0\|_{D(A^{\frac{1}{2}})}\leq \min\left(\nu_\h,\nu_\v\right)\lambda_1^{\frac 14}, $ \eqref{NS}  has a unique global solution.

\item[(3)] Let us note that in the periodic case if the vertical viscosity, $\nu_\v,$ in \eqref{NS} is big enough,
the spectrum of the Laplace on functions with null vertical mean is
pushed far from zero, which ensures the exponential decay of the
linear solution.
So that in order to prove the global well-posedness of \eqref{NS} in $\TT^3,$  we need first to decompose the solution into a 2D part and a 3D
part with vertical null mean, and  then to  obtain that the interactions between the
2D part and the 3D part are small. We shall not present detail here.
\end{itemize}
\end{rmk}

In the case when $\nu_\h=0$ in \eqref{NS}, that is
 \begin{equation}\label{S1eq4}
\begin{cases}
\partial_t u+u\cdot\nabla u-\nu_\v\partial_3^2 u=-\nabla p \quad\mbox{for}\ (t,x)\in \R^+\times\Om,\\
\dive u=0,\\
u|_{t=0}=u_0,\\
u|_{x_3=0}=u|_{x_3=1}=0,
\end{cases}
\end{equation}
where $\Omega=\R^2\times ]0,1[,$ we have the following global well-posedness result for \eqref{S1eq4}:

\begin{thm}\label{thm3}
{\sl Given solenoidal vector field $u_0\in L^2\cap\cB_h^2(\Om),$ there exists a small enough positive constant
$c$ so that if
\beq \label{S1eq5}
\|u_0\|_{L^2}+\|u_0\|_{\cB_\h^2}\leq c\nu_\v,
\eeq
\eqref{S1eq4} has a unique solution
$u\in C_b([0,\infty[; L^2\cap \cB_h^2(\Om))$ with $\p_3u\in L^2(\R^+; L^2\cap \cB_h^2(\Om)).$}
\end{thm}

We remark that the boundary condition in \eqref{S1eq4} is quite natural, which corresponds to  the Dirichlet boundary condition  for  1-D heat equation in
a bounded interval. So far, we still do not know how to prove similar version of Theorem \ref{thm3} in the whole space case.

\smallskip

Let us complete this section by the notations in this context.

Let $A, B$ be two operators, we denote $[A;B]=AB-BA,$ the commutator
between $A$ and $B$. For $a\lesssim b$, we mean that there is a
uniform constant $C,$ which may be different on different lines,
such that $a\leq Cb$.  We denote by $(a|b)$ the $L^2(\R^3)$ inner
product of $a$ and $b,$ $(d_\ell)_{\ell\in\Z}$ will be a generic element of $\ell^1(\Z)$
 so that $\sum_{\ell\in\Z}d_\ell=1$.
For $X$ a Banach space and $I$ an interval of $\R,$ we denote by
${C}(I;\,X)$ the set of continuous functions on $I$ with values in
$X,$  and by ${{C}}_b(I;\,X)$ the subset of bounded
functions of ${C}(I;\,X).$  For $q\in[1,+\infty],$ the
notation $L^q(I;\,X)$ stands for the set of measurable functions on
$I$ with values in $X,$ such that $t\longmapsto\|f(t)\|_{X}$ belongs
to $L^q(I).$

 \setcounter{equation}{0}
\section{ Ideas of the proof and structure of the paper}\label{Sect2}

In what follows, we shall always denote
\beno [f]_\e(t,x_\h,x_3)\eqdefa f(t,x_\h,\e x_3). \eeno

In order to deal with $[\buh_0]_\e$ which appears in \eqref{S1eq3}, motivated by \cite{CG10, CZ5},
we construct  $(\bar{u}^\h, p^\h)$ through
\begin{equation}\label{S2eq2}
\begin{cases}
\partial_t\buh+\buh\cdot\na_\h\buh-\nu_\h\Delta_\h \buh-\nu_{\rm v}\e^2\partial_z^2 \buh=-\na_\h p^\h\quad\mbox{for}\ (t,x_\h,z)\in \R^+\times\R^3,\\
\dive_\h\buh =0,\\
\buh|_{t=0}=\buh_0.
\end{cases}
\end{equation}
By taking $\dive_\h$ to the momentum equation of \eqref{S2eq2}, we find
\beq \label{S2eq5}
p^\h=\sum_{i,j=1}^2(-\D_\h)^{-1}\p_i\p_j(\bar{u}^i\bar{u}^j).
\eeq
To handle $v_0$ in \eqref{S1eq3},  we construct $v_L$ via
\begin{equation}\label{S2eq1}
\begin{cases}
\partial_t v_L-\nu_\h\Delta_\h v_L-\nu_{\rm v}\partial_3^2 v_L=0\quad\mbox{for}\ (t,x)\in \R^+\times\R^3,\\
v_L|_{t=0}=v_0.
\end{cases}
\end{equation}
Then the strategy to the proof of Theorem \ref{thPZ9} will be as follows:
we first write
\beq\label{S2eq3}
u=v_L+[\buh]_\e+R \andf \pi=p-p^\h.
\eeq
It follows form \eqref{NS}, \eqref{S2eq2} and \eqref{S2eq1}  that $(R, \na\pi)$ verifies
\begin{equation}\label{S2eq4}
\begin{cases}
\partial_tR+u\cdot\na R+R\cdot\na(v_L+[\buh]_\e)-\nu_\h\Delta_\h R-\nu_{\rm v}\partial_3^2 R+\na \pi=F,\\
 \with F=(F^\h, F^\v) \andf \\
F^\h=-v_L\cdot\na(v_L^\h+[\buh]_\e)
-[\buh]_\e\cdot\na_\h v_L^\h \andf\\
 F^{\v}=-v_L\cdot\na v_L^3-[\buh]_\e\cdot\na_h v_L^3-\p_3[p^\h]_\e,\\
\dive R =0,\\
R|_{t=0}=0.
\end{cases}
\end{equation}
We are going to prove that \eqref{S2eq4} has a unique global solution under the smallness condition \eqref{S1eq1}.

In order to exploring  the main idea to the proof of Theorem \ref{thPZ9}. Let us first assume $\buh_0=0$ in \eqref{S1eq3}. Moreover, to avoid technicality,
 we assume $v_0\in H^1\cap L^\infty(\R^3).$ Then we have the following
simplified version of Theorem \ref{thPZ9}.

\begin{thm}\label{thm2}
{\sl Let $v_0\in H^{1}\cap L^\infty$ be a solenoidal vector field, we denote
\beq \label{S2eq8}
E(v_0)\eqdefa \|v_0\|_{L^2}^3\|\nabla_\h v_0\|_{L^2}+\|v_0\|^2_{L^\infty}\|v_0\|^2_{L^2}.\eeq  We assume that
 $\nu_\v\geq \nu_\h>0.$  Then there exist two positive constants $c, C$ so that if
\beq\label{S2eq6}
\frac{E(v_0)}{\sqrt{\nu_\h \nu_\v}}\exp\Bigl(\frac{\|v_0\|_{H^{0,1}}^4}{C\nu_\h^4}
\Bigr)\leq
c\nu_\h^3, \eeq
\eqref{NS} with initial data $v_0$ has a unique global solution
$u$ with
$$u\in C_b([0,\infty[; H^{0,1})\hskip0.5cm\text{and}\hskip0.5cm\nabla u\in L^2(\R_+; H^{0,1}).$$
Here $H^{0,1}(\R^3)$ denotes the space of functions $f$ with both $f$ and $\p_3f$ belonging to $L^2(\R^3).$}
\end{thm}

We begin the proof of the above theorem by  the following useful lemmas.

\begin{lem}
\label{S2lem1}
{\sl  Let $a=(a^\h,a_3)$ be a solenoidal vector
field in $H^{0,1}$ with $\nabla_\h a$ belonging to $H^{0,1}$.
Let $b\in H^{0,1}$ with $\nabla_\h b\in H^{0,1}$. Then one has
$$
\bigl(a\cdot\nabla b | b\bigr)_{H^{0,1}}\leq C\|a\|_{H^{0,1}}^{\frac 12}
\|\nabla_\h a\|_{H^{0,1}}^{\frac 12}\|b\|_{H^{0,1}}^{\frac
12}\|\nabla_\h b\|_{H^{0,1}}^{\frac 32}.$$}
\end{lem}

\begin{proof} We recall that $\|a\|_{H^{0,1}}=\bigl(\|a\|_{L^2}^2+\|\p_3a\|_{L^2}^2\bigr)^{\f12}.$
It is easy to observe that due to $\dive a=0,$ $\bigl(a\cdot\nabla b | b\bigr)_{L^2}=0,$ so that there holds
\beno
\begin{split}
\bigl(a\cdot\nabla b | b\bigr)_{H^{0,1}}=&\bigl(a\cdot\nabla b | b\bigr)_{L^2}+\bigl(\partial_3 (a\cdot\nabla b) | \partial_3 b\bigr)_{L^2}\\
=&\bigl(a\cdot\nabla\partial_3 b | \partial_3 b\bigr)_{L^2}+\bigl(\partial_3 a\cdot\nabla b | \partial_3 b\bigr)_{L^2}\\
=&\bigl(\partial_3 a^\h\cdot\nabla_\h b | \partial_3 b\bigr)_{L^2}+\bigl(\partial_3 a_3\partial_3 b | \partial_3 b\bigr)_{L^2}.
\end{split}
\eeno
Notice that
\beno
\|f\|_{L^\infty_\v(L^2_\h)}\lesssim \|f\|_{L^2}^{\f12}\|\p_3 f\|_{L^2}^{\f12},
\eeno
and $\partial_3 a_3=-\dive_\h a^h$, we deduce from the law of product in Sobolev spaces that
\beno
\begin{split}
\mid \bigl(\dive_\h a^\h\partial_3 b | \partial_3 b\bigr)_{L^2}\mid\leq & \|\dive_h a^h\|_{L^\infty_\v(\dH^{-\f12}_\h)}\|(\partial_3 b)^2\|_{L^1_\v(\dH^{\f12}_\h)}\\
\leq & \|\nabla_\h a\|_{L^\infty_\v(\dH^{-\f12}_\h)}\|\p_3 b\|_{L^2_\v(\dH^{\f34}_\h)}^2\\
\lesssim &\|a\|_{L^\infty_\v(\dH^{\f12}_\h)}\|\p_3 b\|_{L^2}^{\f12}\|\na_\h\p_3 b\|_{L^2}^{\f32}.
\end{split}
\eeno
On the other hand, we write
\beno
\f{d}{d x_3}\int_{\R^2}|\xi_\h||\hat{f}(\xi_\h,x_3)|^2\,d\xi_\h=\int_{\R^2}|\xi_\h| \p_3|\hat{f}(\xi_\h,x_3)|^2\,d\xi_\h.
\eeno
Integrating the above equality over $]-\infty, x_3]$ and using H\"older's inequality, we achieve
\beno
\|f\|_{L^\infty_\v(\dH^{\f12}_\h)}\lesssim \|f\|_{L^2}^{\f12}\|\na_\h\p_3f\|_{L^2}^{\f12}.
\eeno
As a result, it comes out
\beq \label{S2eq7}
\mid \bigl(\dive_\h a^\h\partial_3 b | \partial_3 b\bigr)_{L^2}\mid
\lesssim \|a\|_{L^2}^{\f12}\|\na_\h\p_3a\|_{L^2}^{\f12}\|\p_3 b\|_{L^2}^{\f12}\|\na_\h\p_3 b\|_{L^2}^{\f32}.
\eeq
Whereas observing that
\beno
\begin{split}
\mid\bigl(\partial_3 a^\h\cdot\nabla_\h b | \partial_3 b\bigr)_{L^2}\mid\leq &\|\p_3a^\h\|_{L^2_\v(L^4_\h)}\|\na_\h b\|_{L^\infty_\v(L^2_\h)}\|\p_3b\|_{L^2_\v(L^4_\h)}\\
\lesssim &\|\p_3a^\h\|_{L^2}^{\f12}\|\na_\h\p_3a^\h\|_{L^2}^{\f12}\|\na_\h b\|_{L^2}^{\f12}\|\na_\h\p_3 b\|_{L^2}^{\f12}\|\p_3b\|_{L^2}^{\f12}\|\na_\h\p_3b\|_{L^2}^{\f12},
\end{split}
\eeno
which together with \eqref{S2eq7} ensures the lemma.
\end{proof}

The next lemma is concerned with the linear equation \eqref{S2eq1}, which tells us
 the small quantities that will be used in what follows.

\begin{lem}
\label{linear}
{\sl Let $v_0\in L^2$ with $\nabla_\h v_0\in
L^2$. Let $v_L$ be the corresponding solution of \eqref{S2eq1}. Then we have
$$\|\partial_3 v_L\|_{L^2_t(L^2)}^2 \leq
\frac{\|v_0\|_{L^2}^2}{2\nu_\v} \andf \|\partial_3\nabla_\h
v_L\|_{L^2_t(L^2)}^2\leq \frac{\|\nabla_\h v_0\|_{L^2}^2}{2\nu_\v}.$$}
\end{lem}
\begin{proof} Indeed by applying standard
 energy method to \eqref{S2eq1}, we get
\beno
\begin{split}
&\|v_L(t)\|_{L^2}^2+2\nu_\h\int_0^t\|\nabla_\h v_L(t')\|_{L^2}^2\,dt'+2\nu_\v\int_0^t\|\partial_3 v_L(t')\|^2_{L^2}\,dt'=\|v_0\|^2_{L^2},\\
&\|\na_\h v_L(t)\|_{L^2}^2+2\nu_\h\int_0^t\|\nabla_\h^2 v_L(t')\|_{L^2}^2\,dt'+2\nu_\v\int_0^t\|\partial_3 \na_\h v_L(t')\|^2_{L^2}=\|\na_\h v_0\|^2_{L^2},
\end{split}
\eeno
which implies the lemma.
\end{proof}

\begin{lem}\label{S2lem3}
{\sl Let $E(v_0)$ be given by \eqref{S2eq8}. Then under the assumptions of Lemma \ref{linear} and $\nu_\v\geq\nu_\h>0,$ one has
$$\int_0^\infty\|v_L\otimes v_L(t)\|_{H^{0,1}}^2\,dt\lesssim \frac{E(v_0)}{\sqrt{\nu_\h\nu_\v}}.$$}
\end{lem}

\begin{proof}
We begin by writing that
\beq \label{S2eq10}
\begin{split}
\int_0^\infty\|v_L\otimes v_L(t)\|_{H^{0,1}}^2\,dt=&\int_0^\infty\bigl(\|v_L\otimes v_L(t)\|_{L^2}^2
+\|\partial_3(v_L\otimes v_L)(t)\|^2_{L^2}\bigr)\,dt\\
\leq & C\int_0^\infty\bigl(\|v_L(t)\|_{L^4}^4+\|v_L\partial_3 v_L(t)\|_{L^2}^2\bigr)\,dt.
\end{split}
\eeq
Applying Lemma \ref{linear} and maximal principle for \eqref{S2eq1} gives
\beq \label{S2eq9}
\begin{split}
\int_0^\infty\|v_L\partial_3 v_L(t)\|_{L^2}^2\,dt\leq &\int_0^\infty
\|v_L(t)\|^2_{L^\infty}\|\partial_3 v_L(t)\|^2_{L^2}\,dt\\
\leq &
 \frac{\|v_0\|^2_{L^\infty}\|v_0\|_{L^2}^2}{2\nu_\v}.
 \end{split}
 \eeq

To handle the other term in \eqref{S2eq10},
by applying the Sobolev imbedding of $\dot H^{\frac
14}(\R_\v)\hookrightarrow L^4(\R_\v),$ we obtain
$$\|v_L(t,x_\h,\cdot)\|_{L^4(\R_\v)}\leq C\|v_L(t,x_\h,\cdot)\|_{L^2_\v}^{\frac 34}
\|\partial_3 v_L(t,x_\h,\cdot)\|^{\frac 14}_{L^2_\v},$$
which together with Lemma \ref{linear} implies that
\beno
\begin{split}
\int_0^\infty\|v_L(t)\|_{L^4}^4\,dt\lesssim & \int_0^\infty\int_{\R^2_\h}\|v_L(t,x_\h,\cdot)\|_{L^2_\v}^3\|\partial_3 v_L(t,x_\h,\cdot)\|_{L^2_\v}\,dx_\h\,dt\\
\lesssim &\|v_L\|_{L^6(\R^+; L^6_\h(L^2_\v))}^3\|\partial_3 v_L\|_{L^2_t(L^2)}^2\\
 \lesssim & \nu_\v^{-\f12}\|v_L\|_{L^6(\R^+; L^6_\h(L^2_\v))}^3 \|v_0\|_{L^2}.
 \end{split}
 \eeno
Whereas noticing  from
Sobolev imbedding Theorem that $\dot H^{\frac 23}(\R^2_h)\hookrightarrow
L^6(\R^2_h)$, we  write
$$\|v_L(t,\cdot,x_3)\|_{L^6_\h}\leq C\|v_L(t,\cdot,x_3)\|_{L^2_\h}^{\frac 13}\|\nabla_\h v_L(t,\cdot,x_3)\|_{L^2_\h}^{\frac 23}.$$
Taking the $L^2_\v$ norm leads to
$$\|v_L(t)\|_{L^6_\h(L^2_\v)}\leq
\|v_L(t)\|_{L^2_\v(L^6_\h)}\leq C\|v_L(t)\|_{L^2}^{\frac 13}\|\nabla_\h
v_L(t)\|_{L^2}^{\frac 23},$$
which together with Lemma \ref{linear} ensures that
\beno
\begin{split}
\|v_L\|_{L^6(\R^+; L^2_\v(L^6_\h))}^3\lesssim & \|v_L\|_{L^\infty(\R^+; L^2)}\|\nabla_\h v_L\|_{L^\infty(\R^+;L^2)}\|\nabla_\h v_L\|_{L^2(\R^+;L^2)}\\
\lesssim &\nu_\h^{-\f12}\|v_0\|_{L^2}^2\|\na_\h v_0\|_{L^2}.
\end{split}
\eeno
This gives rise to
\beno
\int_0^\infty\|v_L(t)\|_{L^4}^4\,dt\lesssim \frac{\|v_0\|_{L^2}^3\|\nabla_h v_0\|_{L^2}}{\sqrt{\nu_\h\nu_\v}}.
\eeno
Along with \eqref{S2eq10} and  \eqref{S2eq9}, we complete the proof of the lemma.
\end{proof}

\begin{proof}[Proof of the Theorem \ref{thm2}] Let $v_L$ be determined by \eqref{S2eq1}. We write
$$u=v_L+R. $$
Inserting the above substitution into \eqref{NS} yields
\begin{equation}\label{S2eq11}
\begin{cases}
\partial_t R+R\cdot\nabla R-\nu_\h\Delta_\h R-\nu_\v\partial_3^2 R+\nabla
p=-R\cdot\nabla v_L-v_L\cdot\nabla R-v_L\cdot\nabla v_L,\\
\dive R=0,\\
R(0)=0.
\end{cases}
\end{equation} It follows from classical theory on Navier-Stokes system that \eqref{S2eq11} has a unique solution
$R\in C([0,T^\ast[; H^{0,1})$ with $\na R\in  L^2(]0,T^\ast[; H^{0,1})$ for some maximal existing time $T^\ast.$ In the following,
we are going to prove that $T^\ast=\infty$ under the smallness condition \eqref{S2eq6}. For simplicity, we just present the {\it a priori} estimate.

 By taking $H^{0,1}$
scalar product of the $R$ equation of  (\ref{S2eq11}) with $R,$  we obtain
\beq \label{S2eq12}
\begin{split}
\frac{1}{2}\frac{d}{dt}\|R\|^2_{H^{0,1}}+&\nu_\h\|\nabla_\h R\|^2_{H^{0,1}}+\nu_\v\|\partial_3 R\|^2_{H^{0,1}}=- \bigl(R\cdot\nabla R | R\bigr)_{H^{0,1}}\\
&-\bigl(v_L\cdot\nabla R | R\bigr)_{H^{0,1}}-\bigl(R\cdot\nabla v_L | R\bigr)_{H^{0,1}}-
\bigl(v_L\cdot\nabla v_L | R\bigr)_{H^{0,1}}.
\end{split}
\eeq
Applying Lemma \ref{S2lem1} gives
\beno
\mid\bigl(R\cdot\nabla R | R\bigr)_{H^{0,1}}\mid \leq \|R\|_{H^{0,1}}\|\nabla_\h R\|^2_{H^{0,1}},
\eeno
and
\beno
\mid\bigl(v_L\cdot\nabla R | R\bigr)_{H^{0,1}}\mid\leq C\|v_L\|_{H^{0,1}}^{1/2}\|\nabla_\h v_L\|_{H^{0,1}}^{1/2}\|R\|_{H^{0,1}}^{1/2}\|\nabla_\h R\|^{3/2}_{H^{0,1}}.
\eeno
Then by applying Young's inequality, $xy\leq \frac 34 x^{\frac 43}+\frac 14 y^4,$ we achieve
$$
\mid\bigl(v_L\cdot\nabla R | R\bigr)_{H^{0,1}}\mid\leq {C}\nu_\h^{-3}\|v_L\|^2_{H^{0,1}}\|\nabla_\h v_L\|^2_{H^{0,1}}\|R\|^2_{H^{0,1}}
+\frac{\nu_\h}{100}\|\nabla_h R\|^2_{H^{0,1}}.$$
Similarly, notice that  $H^{\frac 12}(\R^2_\h)\subset L^4(\R^2_\h),$ we get
\beno
\begin{split}
\mid\bigl(R\cdot\nabla v_L | R\bigr)_{H^{0,1}}\mid\leq &
\|R\|^2_{L^4_\h(H^1_\v)}\|\nabla v_L\|_{H^{0,1}}\\
\leq & \|R\|_{H^{0,1}}\|\nabla_\h R\|_{H^{0,1}}\|\nabla v_L\|_{H^{0,1}}\\
\leq &{C}{\nu_\h}^{-1}\|\nabla v_L\|^2_{H^{0,1}}\|R\|^2_{H^{0,1}}+\frac{\nu_\h}{100}\|\nabla_\h R\|_{H^{0,1}}^2.
\end{split}
\eeno
For the last term in \eqref{S2eq11}, we first get, by using integrating by parts, that
\beno
\bigl(v_L\cdot\nabla v_L | R\bigr)_{H^{0,1}}=-\bigl(v_L\otimes v_L | \na R\bigr)_{H^{0,1}},
\eeno
which implies
\beno
\begin{split}
\mid\bigl(v_L\cdot\nabla v_L | R\bigr)_{H^{0,1}}\mid\leq& \|v_L\otimes v_L\|_{H^{0,1}}\|\na R\|_{H^{0,1}}\\
\leq &C\nu_\h^{-1}\|v_L\otimes v_L\|_{H^{0,1}}^2+\f{\nu_\h}{100}\|\na R\|_{H^{0,1}}.
\end{split}
\eeno

Let us denote
\beq \label{S2eq13}
T^\star\eqdefa \sup\bigl\{ \ T<T^\ast,\ \|R\|_{L^\infty_T(H^{0,1})}\leq \f{\nu_\h}4\ \bigr\}.
\eeq We are going to prove that $T^\star=T^\ast.$ Otherwise
by substituting the above estimates into \eqref{S2eq12},  for $t\leq T^\star,$ we arrive at
\beq \label{S2eq18}
\begin{split}
\frac{d}{dt}\|R\|^2_{H^{0,1}}+&\nu_\h\|\nabla_\h R\|^2_{H^{0,1}}+\nu_\v\|\partial_3 R\|^2_{H^{0,1}}\leq C\nu_\h^{-1}\|v_L\otimes v_L\|_{H^{0,1}}^2\\
&+C\Bigl({\nu_\h}^{-1}\|\nabla v_L\|^2_{H^{0,1}}+\nu_\h^{-3}\|v_L\|^2_{H^{0,1}}\|\nabla_\h v_L\|^2_{H^{0,1}}\Bigr)\|R\|^2_{H^{0,1}}.
\end{split}
\eeq
Applying  Gronwall's inequality to \eqref{S2eq18} yields
\beno
\begin{split}
\|R\|^2_{L^\infty_t(H^{0,1})}\leq & C\nu_\h^{-1}\|v_L\otimes v_L\|_{L^2_t(H^{0,1})}^2\\
&\times\exp \Bigl(\nu_\h^{-1}\|\nabla v_L\|_{L^2_t(H^{0,1})}^2+\nu_\h^{-3}\|v_L\|^2_{L^\infty_t(H^{0,1})}\|\nabla_\h v_L\|^2_{L^2_t(H^{0,1})}\Bigr),
\end{split}
\eeno
from which, Lemmas \ref{linear} and \ref{S2lem3}, for $t\leq T^\star,$ we infer
\beno
\|R\|_{L^\infty_t(H^{0,1})}^2\leq \nu_\h^{-\f32}\nu_\v^{-\f12} E(v_0)\exp\left(C\nu_\h^{-4}\|v_0\|_{H^{0,1}}^4+C\nu_\h^{-2}\|v_0\|^2_{H^{0,1}}\right).\eeno
Then under the smallness condition \eqref{S2eq6}, we have
\beq\label{S2eq14}
\|R\|_{L^\infty_t(H^{0,1})}\leq\f{\nu_\h}8\quad\mbox{for}\ \ t\leq T^\star,
\eeq
which contradicts with \eqref{S2eq13}. This in turn shows that $T^\star=T^\ast=\infty.$ Furthermore inserting the estimate
\eqref{S2eq14} into \eqref{S2eq18} shows that $\na R\in L^2(\R^+; H^{0,1}).$ This completes the proof of Theorem \ref{thm2}.
\end{proof}

The organization of this paper is as follows:

In the third section, we shall recall some basic facts on Littlewood-Paley theory;

 In the fourth section, we present the {\it priori} estimates for smooth enough solutions of \eqref{S2eq2} and \eqref{S2eq1};

 In the fifth section, we prove Theorem \refer{thPZ9};

 In the sixth section, we present the proof of Theorem \ref{thm3};

 Finally in the Appendix, we present the proofs of several technical lemmas which have been used in the proof of   Theorem \refer{thPZ9}.

 \setcounter{equation}{0}
\section{Basics on Littlewood-Paley theory}\label{Sect3}

Before we present the function spaces we are going to work with in
this context, let us briefly recall some basic facts on
Littlewood-Paley theory (see e.g. \cite{bcd}). Let $\varphi$ and
$\chi$ be smooth functions supported in $\mathcal{C}\eqdefa \{
\tau\in\R^+,\ \frac{3}{4}\leq\tau\leq\frac{8}{3}\}$ and
$\frak{B}\eqdefa \{ \tau\in\R^+,\ \tau\leq\frac{4}{3}\}$
respectively such that
\begin{equation*}
 \sum_{j\in\Z}\varphi(2^{-j}\tau)=1 \quad\hbox{for}\quad \tau>0\quad\mbox{and}\quad  \chi(\tau)+ \sum_{j\geq
0}\varphi(2^{-j}\tau)=1\quad\hbox{for}\quad \tau\geq 0.
\end{equation*}
For $a\in{\mathcal S}'(\R^3),$ we set \beq
\begin{split}
&\Delta_k^\h
a\eqdefa\cF^{-1}(\varphi(2^{-k}|\xi_\h|)\widehat{a}),\qquad
S^\h_ka\eqdefa\cF^{-1}(\chi(2^{-k}|\xi_\h|)\widehat{a}),
\\
& \Delta_\ell^\v a
\eqdefa\cF^{-1}(\varphi(2^{-\ell}|\xi_3|)\widehat{a}),\qquad \
S^\v_\ell a \eqdefa \cF^{-1}(\chi(2^{-\ell}|\xi_3|)\widehat{a}),
 \quad\mbox{and}\\
&\Delta_ja\eqdefa\cF^{-1}(\varphi(2^{-j}|\xi|)\widehat{a}),
 \qquad\ \ \
S_ja\eqdefa \cF^{-1}(\chi(2^{-j}|\xi|)\widehat{a}), \end{split}
\label{1.0}\eeq where  $\xi_\h=(\xi_1,\xi_2),$ $\xi=(\xi_\h,\xi_3),$
$\cF a$ and $\widehat{a}$ denote the Fourier transform of the
distribution $a.$ The dyadic operators satisfy the property of
almost orthogonality:
\begin{equation}\label{C4}
\Delta_k\Delta_j a\equiv 0 \quad\mbox{if}\quad| k-j|\geq 2
\quad\mbox{and}\quad \Delta_k( S_{j-1}a \Delta_j b) \equiv
0\quad\mbox{if}\quad| k-j|\geq 5.
\end{equation}
Similar properties hold for $\D_k^\h$ and $\D_\ell^\v.$

Let us recall the  anisotropic
Bernstein type lemma from \cite{CZ5, Paicu05}\,.

\begin{lem}\label{lemBern}
{\sl Let $\cB_{\h}$ (resp.~$\cB_{\rm v}$) a ball
of ~$\R^2_{\h}$ (resp.~$\R_{\rm v}$), and~$\cC_{\h}$ (resp.~$\cC_{\rm v}$) a
ring of~$\R^2_{\h}$ (resp.~$\R_{\rm
v}$); let~$1\leq p_2\leq p_1\leq
\infty$ and ~$1\leq q_2\leq q_1\leq \infty.$ Then there hold
\beno
\begin{split}
\mbox{if}\ \ \Supp \wh a\subset 2^k\cB_{\h}&\Rightarrow
\|\partial_{x_{\rm h}}^\alpha a\|_{L^{p_1}_{\rm h}(L^{q_1}_{\rm v})}
\lesssim 2^{k\left(|\alpha|+2\left(1/{p_2}-1/{p_1}\right)\right)}
\|a\|_{L^{p_2}_{\rm h}(L^{q_1}_{\rm v})};\\
\mbox{if}\ \ \Supp\wh a\subset 2^\ell\cB_{\rm v}&\Rightarrow
\|\partial_{x_3}^\beta a\|_{L^{p_1}_{\rm h}(L^{q_1}_{\rm v})}
\lesssim 2^{\ell\left(\beta+(1/{q_2}-1/{q_1})\right)} \|
a\|_{L^{p_1}_{\rm h}(L^{q_2}_{\rm v})};\\
\mbox{if}\ \ \Supp\wh a\subset 2^k\cC_{\h}&\Rightarrow
\|a\|_{L^{p_1}_{\rm h}(L^{q_1}_{\rm v})} \lesssim
2^{-kN}\sup_{|\alpha|=N}
\|\partial_{x_{\rm h}}^\alpha a\|_{L^{p_1}_{\rm
h}(L^{q_1}_{\rm v})};\\
\mbox{if}\ \ \Supp\wh a\subset2^\ell\cC_{\rm v}&\Rightarrow
\|a\|_{L^{p_1}_{\rm h}(L^{q_1}_{\rm v})} \lesssim 2^{-\ell N}
\|\partial_{x_3}^N a\|_{L^{p_1}_{\rm h}(L^{q_1}_{\rm v})}.\end{split}\eeno
}
\end{lem}

Due to the anisotropic spectral properties of the linear
equation \eqref{S2eq1}, we need  the following anisotropic type Besov norm:

\begin{defi}\label{def2}
{\sl  Let  $s_1,s_2\in\R$ and
$a\in{\mathcal S}_h'(\R^3),$ we define the norm
$$
\|a\|_{\cB^{s_1,s_2}}\eqdefa \Bigl(2^{\ell s_2} \bigl(2^{k
s_1}\|\D_{k}^\h\D^\v_{\ell}a\|_{L^2}\bigr)_{\ell^{2}}\Bigr)_{\ell^{1}}.
$$
In particular, when $s_1=0$ we denote $\cB^{0,s_2}$ by $B^{0,s_2}$ with $$
\|a\|_{B^{0,s_2}}\eqdefa \sum_{\ell\in\Z}2^{\ell s_2}\|\D_\ell^\v a\|_{L^2}. $$ }
\end{defi}

We recall  the classical homogeneous anisotropic Sobolev norm as follows
\beno
\|a\|_{\dH^{s_1,s_2}}\eqdefa \Bigl(\sum_{(k,\ell)\in\in\Z^2} 2^{2k
s_1} 2^{2\ell s_2}\|\D_{k}^\h\D^\v_{\ell}a\|_{L^2}^2\Bigr)^{\f12}.
\eeno

In  order to obtain a better description of the regularizing effect
for the transport-diffusion equation, we will use Chemin-Lerner type
spaces.

\begin{defi}\label{def3}
{\sl Let  $p\in[1,+\infty]$ and $T\in ]0,+\infty]$.
 We define the norms of  $\wt{L}^p_T(\cB^{s_1,s_2}(\R^3))$ and $\wt{L}^p_T(B^{0,s_2})$  by
\beno \|a\|_{\wt{L}^p_T(\cB^{s_1,s_2})}\eqdefa \Bigl(2^{\ell s_2} \bigl(2^{k
s_1}\|\D_{k}^\h\D^\v_{\ell}a\|_{L^p_T(L^2)}\bigr)_{\ell^{2}}\Bigr)_{\ell^{1}},
\eeno  and
$$
\|a\|_{\wt{L}^p_T(B^{0,s_2})}\eqdefa \sum_{\ell\in\Z}2^{\ell s_2}\|\D_\ell^\v a\|_{L^p_T(L^2)} $$
 respectively. }
\end{defi}

In particular, when $p=2,$ we have
\beq\label{Def3eq1}
\begin{split}
\|a\|_{\wt{L}^2_T(\cB^{0,s_2})}=&\sum_{\ell\in\Z}2^{\ell s_2}\bigl(\sum_{k\in\Z}\|\D_k^\h\D_\ell^\v a\|_{L^2_T(L^2)}^2\bigr)^{\f12}\\
=&\sum_{\ell\in\Z}2^{\ell s_2}\|\D_\ell^\v a\|_{L^2_T(L^2)}=\|a\|_{\wt{L}^2_T(B^{0,s_2})}.
\end{split}
\eeq

In order to study  fluid evolving between two parallel plans, namely to prove Theorem \ref{thm3}, we also need the following norms:
\begin{defi}\label{def4}
{\sl Let $\Omega=\R^2\times ]0,1[,$  $s\in\R$ and $p\in [1,\infty],$ for $a\in C^\infty(\Om),$ we define
\beq\label{def4eq}
\|a\|_{\cB_\h^s(\Om)}\eqdefa \sum_{k\in\Z}2^{ks}\|\D_k^\h a\|_{L^2(\Om)} \andf \|a\|_{\wt{L}^p_T(\cB_\h^s(\Om))}\eqdefa \sum_{k\in\Z}2^{ks}\|\D_k^\h a\|_{L^p_T(L^2(\Om))}.
\eeq}
\end{defi}

To overcome the difficulty that one can not use Gronwall's type  argument for the Chemin-Lerner type norms,
we  need the time-weighted Chemin-Lerner norm introduced by the authors in \cite{PZ1}:

\begin{defi}\label{defpz}
Let $f(t)\in L^1_{loc}(\R_+)$, $f(t)\geq 0$. We define
$$\|u\|_{\widetilde{L^2_{T,f}}(B^{0,s_2})}=\sum_{\ell\in\Z} 2^{\ell s_2}
\Bigl(\int_0^Tf(t)\|\Delta_\ell^\v u(t)\|_{L^2}^2dt\Bigr)^{\frac 12}.$$
\end{defi}

Finally  we  recall the isentropic para-differential
decomposition from \cite{Bo}: let $a$ and~Ê$b$ be in~$ \cS'(\R^3)$,
\beq
\label{pd}\begin{split} &
ab=T_ab+{T}_b a+ R(a,b)\quad\mbox{or} \quad ab=T_ab+ \bar{R}(a,b)  \quad\hbox{where}\\
& T_ab\eqdefa \sum_{j\in\Z}S_{j-1}a\Delta_jb,\quad \bar{R}(a,b)\eqdefa\sum_{j\in\Z}\Delta_ja S_{j+2}b
 \andf\\
&R(a,b)\eqdefa \sum_{j\in\Z}\Delta_ja\tilde{\Delta}_{j}b,\quad\hbox{with}\quad
\tilde{\Delta}_{j}b\eqdefa \sum_{\ell=j-1}^{j+1}\D_\ell a. \end{split} \eeq
In what follows, we shall use the  anisotropic version of
Bony's decomposition for both horizontal and vertical variables.

As an application of the above basic facts on Littlewood-Paley
theory, we present the following product law in the anisotropic
Besov spaces.

\begin{lem}\label{S3lem0}
{\sl Let  $\tau_1,\tau_2\in ]-1/2, 1/2]$ with  $\tau_1+\tau_2>0.$
Let $a\in B^{0,\tau_1}$ with $\na_\h a\in B^{0,\tau_1},$ let $b\in B^{0,\tau_2}$ with $\na_\h b\in B^{0,\tau_2}.$ Then one has
\beq \label{S3eq19}
\|ab\|_{B^{0,\tau_1+\tau_2-\f12}}\lesssim \|a\|_{B^{0,\tau_1}}^{\f12}\|\na_\h a\|_{B^{0,\tau_1}}^{\f12}\|b\|_{B^{0,\tau_2}}^{\f12}\|\na_\h b\|_{B^{0,\tau_2}}^{\f12}.
\eeq}
\end{lem}

\begin{proof} We first get, by applying Bony's decomposition to $ab$ in  vertical variable, that
\beq \label{S3eq20}
ab=T^\v_a b+{T}^\v_ba+R^\v(a,b).
\eeq
Due to $\tau_1\leq \f12,$ applying Lemma \ref{lemBern} gives
\beno\begin{split}
\|S_\ell^\v a\|_{L^\infty_\v(L^4_\h)}\lesssim & \sum_{\ell'\leq\ell-1}2^{\f{\ell'}2}\|\D_{\ell'}^\v a\|_{L^2_\v(L^4_\h)}\\
\lesssim & \sum_{\ell'\leq\ell-1}2^{\f{\ell'}2}\|\D_{\ell'}^\v a\|_{L^2}^{\f12}\|\D_{\ell'}^\v \na_\h a\|_{L^2}^{\f12}\\
\lesssim& 2^{\ell\left(\f12-\tau_1\right)}\|a\|_{B^{0,\tau_1}}^{\f12}\|\na_\h a\|_{B^{0,\tau_1}}^{\f12}.
\end{split}
\eeno
from which, and the support properties to the Fourier transform of the terms in $T^\v_ab,$ we infer
\beno
\begin{split}
\|\D_\ell^\v(T^\v_ab)\|_{L^2}\lesssim &\sum_{ |\ell'-\ell|\leq 4}\|S_{\ell'-1}^\v a\|_{L^\infty_\v(L^4_\h)}
\|{\D}_{\ell'}^\v b\|_{L^2_\v(L^4_\h)}\\
\lesssim & \sum_{ |\ell'-\ell|\leq 4}\|S_{\ell'-1}^\v a\|_{L^\infty_\v(L^4_\h)}
\|{\D}_{\ell'}^\v b\|_{L^2}^{\f12}\|{\D}_{\ell'}^\v \na_\h b\|_{L^2}^{\f12}\\
\lesssim &d_{\ell} 2^{-\ell(\tau_1+\tau_2-\f12)}\|a\|_{B^{0,\tau_1}}^{\f12}\|\na_\h a\|_{B^{0,\tau_1}}^{\f12}\|b\|_{B^{0,\tau_2}}^{\f12}\|\na_\h b\|_{B^{0,\tau_2}}^{\f12}.
\end{split}
\eeno Here and in all that follows, we always denote $\left(d_\ell\right)_{\ell\in\Z}$ to be a generic element of $\ell^1(\Z)$ so that $\sum_{\ell\in\Z}d_\ell=1.$
The same estimate holds for $T^\v_ba.$

On the other hand, we deduce from  Lemma \ref{lemBern} that
\beno
\begin{split}
\|\D_\ell^\v (R^\v(a,b))\|_{L^2}\lesssim & 2^{\f{\ell}2}\sum_{ \ell'\geq \ell-3}\|\D_{\ell'}^\v a\|_{L^2_\v(L^4_\h)}
\|\wt{\D}_{\ell'}^\v b\|_{L^2_\v(L^4_\h)}\\
\lesssim & 2^{\f{\ell}2}\sum_{ \ell'\geq \ell-3}\|\D_{\ell'}^\v a\|_{L^2}^{\f12}\|\D_{\ell'}^\v \na_\h a\|_{L^2}^{\f12}
\|\wt{\D}_{\ell'}^\v b\|_{L^2}^{\f12}\|\wt{\D}_{\ell'}^\v \na_\h b\|_{L^2}^{\f12}\\
\lesssim & 2^{\f{\ell}2}\sum_{ \ell'\geq \ell-3}d_{\ell'}2^{-\ell'(\tau_1+\tau_2)}\|a\|_{B^{0,\tau_1}}^{\f12}\|\na_\h a\|_{B^{0,\tau_1}}^{\f12}\|b\|_{B^{0,\tau_2}}^{\f12}\|\na_\h b\|_{B^{0,\tau_2}}^{\f12}\\
\lesssim &d_{\ell} 2^{-\ell(\tau_1+\tau_2-\f12)}\|a\|_{B^{0,\tau_1}}^{\f12}\|\na_\h a\|_{B^{0,\tau_1}}^{\f12}\|b\|_{B^{0,\tau_2}}^{\f12}\|\na_\h b\|_{B^{0,\tau_2}}^{\f12},
\end{split}
\eeno where in the last step, we used the fact that $\tau_1+\tau_2>0,$ so that $$\sum_{\ell'\geq \ell-3}d_{\ell'}2^{-\ell'(\tau_1+\tau_2)}\lesssim d_{\ell}2^{-\ell(\tau_1+\tau_2)}.$$
 This completes the proof of the lemma.
\end{proof}

\begin{rmk}\label{S3rmk1}
We remark that the law of product \eqref{S3eq19} works also for Chemin-Lerner norms.
\end{rmk}

 \setcounter{equation}{0}
\section{The {\it a priori} estimates of $v_L$ and $\buh$ }\label{Sect4}

The goal of this section is to present the {\it a priori} estimates of $v_L$ and $\buh.$

\begin{prop}\label{S3prop1}
{\sl Let $v_0\in \cB^{\f12,0}\cap B^{0,\f12}$ and $v_L$ be the corresponding solution of \eqref{S2eq1}. Then we have
\beq \label{S3eq1} \begin{split}
&\|\na_\h v_L\|_{\wt{L}^2_t(B^{0,\f12})}\lesssim \left(\nu_\h\nu_\v\right)^{-\f14}\|v_0\|_{\cB^{\f12,0}},\\
\|v_L\|_{\wt{L}^\infty_t(B^{0,\f12})}+&\nu_\h^{\f12}\|\na_\h v_L\|_{\wt{L}^2_t(B^{0,\f12})}+\nu_\v^{\f12}\|\p_3 v_L\|_{\wt{L}^2_t(B^{0,\f12})}\lesssim \|v_0\|_{B^{0,\f12}}.
\end{split}
\eeq
}
\end{prop}

\begin{proof} We get, by first applying the operator $\D_k^\h\D_\ell^\v$ to the system \eqref{S2eq1} and then taking
$L^2$ inner product of the resulting equation with $\D_k^\h\D_\ell^\v v_L,$ that
\beno
\f12\f{d}{dt}\|\D_k^\h\D_\ell^\v v_L(t)\|_{L^2}^2+\nu_\h\|\D_k^\h\D_\ell^\v\na_\h v_L\|_{L^2}^2+\nu_\v\|\D_k^\h\D_\ell^\v \p_3 v_L\|_{L^2}^2=0.
\eeno
Integrating the above equality over $[0,t]$ and  then taking square root of the resulting equality, we write
\beno
\|\D_k^\h\D_\ell^\v v_L\|_{L^\infty_t(L^2)}+\sqrt{\nu_\h}\|\D_k^\h\D_\ell^\v\na_\h v_L\|_{L^2_t(L^2)}+\sqrt{\nu_\v}\|\D_k^\h\D_\ell^\v \p_3 v_L\|_{L^2_t(L^2)}
\leq \|\D_k^\h\D_\ell^\v v_0\|_{L^2}.
\eeno
By multiplying $2^{\f{k}2}$ to the above inequality and taking $\ell^2$ norm with respect to $k\in\Z$ and then taking $\ell^1$ norm with respect to $\ell\in\Z,$
we achieve
\beq\label{S3eq2}
\|v_L\|_{\wt{L}^\infty_t(\cB^{\f12,0})}+\sqrt{\nu_\h}\|\na_\h v_L\|_{\wt{L}^2_t(\cB^{\f12,0})}+\sqrt{\nu_\v}\|\p_3 v_L\|_{\wt{L}^2_t(\cB^{\f12,0})}
\leq \|v_0\|_{\cB^{\f12,0}}.
\eeq
Whereas it follows from Fourier-Plancherel equality and Lemma \ref{lemBern} that
\beno
\begin{split}
2^\ell\|\D_\ell^\v \na_\h v_L\|_{L^2_t(L^2)}^2=&2^\ell \sum_{k\in\Z}\|\D_k^\h\D_\ell^\v \na_\h v_L\|_{L^2_t(L^2)}^2\\
\lesssim &2^\ell \sum_{k\in\Z}2^{2k}\|\D_k^\h\D_\ell^\v  v_L\|_{L^2_t(L^2)}^2\\
\lesssim &\Bigl(\sum_{k\in\Z}2^{3k}\|\D_k^\h\D_\ell^\v  v_L\|_{L^2_t(L^2)}^2\Bigr)^{\f12}
\Bigl(\sum_{k\in\Z}2^{k}2^{2\ell}\|\D_k^\h\D_\ell^\v  v_L\|_{L^2_t(L^2)}^2\Bigr)^{\f12},
\end{split}
\eeno
from which and \eqref{S3eq2}, we infer
\beno
\begin{split}
\|\na_\h v_L\|_{\wt{L}^2_t(B^{0,\f12})}\lesssim &\sum_{\ell\in\Z} 2^{\f{\ell}2}\|\D_\ell^\v \na_\h v_L\|_{L^2_t(L^2)}\\
\lesssim &\biggl(\sum_{\ell\in\Z}\Bigl(\sum_{k\in\Z}2^{k}\|\D_k^\h\D_\ell^\v \na_\h v_L\|_{L^2_t(L^2)}^2\Bigr)^{\f12}\biggr)^{\f12}\\
&\times \biggl(\sum_{\ell\in\Z}\Bigl(\sum_{k\in\Z}2^{k}\|\D_k^\h\D_\ell^\v  \p_3 v_L\|_{L^2_t(L^2)}^2\Bigr)^{\f12}\biggr)^{\f12}\\
\lesssim& (\nu_\h\nu_\v)^{-\f14}\bigl(\sqrt{\nu_\h}\|\na_\h v_L\|_{\wt{L}^2_t(\cB^{\f12,0})}\bigr)^{\f12}
\bigl(\sqrt{\nu_\v}\|\p_3 v_L\|_{\wt{L}^2_t(\cB^{\f12,0})}\bigr)^{\f12}
\\
\lesssim &(\nu_\h\nu_\v)^{-\f14}\|v_0\|_{\cB^{\f12,0}}.
\end{split}
\eeno This leads to the first  inequality of \eqref{S3eq1}.

On the other hand, we get, by first applying the operator $\D_\ell^\v$ to the system \eqref{S2eq1} and then taking
$L^2$ inner product of the resulting equation with $2^\ell\D_\ell^\v v_L,$ that
\beno
2^{\ell-1}\f{d}{dt}\|\D_\ell^\v v_L(t)\|_{L^2}^2+\nu_\h2^\ell\|\na_\h \D_\ell^\v v_L\|_{L^2}^2+\nu_\v2^\ell\|\p_3\D_\ell^\v v_L\|_{L^2}^2=0.
\eeno
Integrating the above equality over $[0,t]$ and taking square root of the resulting equality, and then taking
$\ell^1$ norm with respect to $\ell\in \Z,$ we obtain the second  inequality of \eqref{S3eq1}. This completes the proof of the proposition.
\end{proof}

\begin{lem}\label{S3lem1}
{\sl Let $\buh_0$ and $\na_\h\buh_0$ be in $L^2(\R^3)\cap L^\infty(\R_\v;L^2(\R^2_\h)).$ Then \eqref{S2eq2} has a unique global solution so that
\beq \label{S3eq5a}
\|\buh\|_{L^\infty_t(L^\infty_\v(L^2_\h))}\leq \|\buh_0\|_{L^\infty_\v(L^2_\h)}.
\eeq
If moreover, $\buh_0\in L^\infty(\R_\v; \dH^{-\d}(\R^2_\h))$ for some $\d\in ]0,1[,$ then we have
\beq \label{S3eq6}
\begin{split}
\int_0^\infty\|\na_\h \buh(t)\|_{L^\infty_\v(L^2_\h)}^2\,dt\leq &  A_{\nu_\h,\d}(\buh_0) \with\\
 A_{\nu_\h,\d}(\buh_0)=&C_\d \nu_\h^{-1}\exp\bigl(C_\d\nu_\h^{-2}\|\buh_0\|_{L^\infty_\v(L^2_\h)}^2(1+\nu_\h^{-2}\|\buh_0\|_{L^\infty_\v(L^2_\h)}^2)\bigr)\\
&\times\biggl(\f{\|\na_\h \buh_0\|_{L^\infty_\v(L^2_\h)}^2\|\buh_0\|_{L^\infty_\v((\dB^{-\d}_{2,\infty})_\h)}^{\f2\d} }{\|\buh_0\|_{L^\infty_\v(L^2_\h)}^{\f2\d}}
+\|\buh_0\|_{L^\infty_\v(L^2_\h)}^2\biggr).
\end{split}
\eeq}
\end{lem}

\begin{proof}   Theorem 1.2 of \cite{CZ5} ensures the global existence of solutions to \eqref{S2eq2}. Moreover, (2.4)
of \cite{CZ5} gives \eqref{S3eq5a}. To prove the
estimate \eqref{S3eq6},
we introduce
\beq \label{S3eq3}
w^\h(t,x_\h,z)\eqdefa \nu_\h^{-\f12}\buh(t, \nu_\h^{\f12}x_\h,z) \andf q^\h(t,x_\h,z)\eqdefa \nu_\h^{-1}p^\h(t, \nu_\h^{\f12}x_\h, z).
\eeq
Then in view of \eqref{S2eq2}, we write
\begin{equation}\label{S3eq4}
\begin{cases}
\partial_t w^\h+w^\h\cdot\na_\h w^\h-\Delta_\h w^\h-{\nu_{\rm v}}\e^2\partial_z^2 w^\h=-\na_\h q^\h\quad\mbox{for}\ (t,x)\in \R^+\times\R^3,\\
\dive_\h w^\h =0,\\
w^\h|_{t=0}=w_0^\h=\nu_\h^{-\f12}\buh_0(\nu_\h^{\f12}x_\h,z).
\end{cases}
\end{equation}
It follows from Theorem 1.2 of \cite{CZ5} that
\beq \label{S3eq5}
\begin{split}
\int_0^\infty\|\na_\h w^\h(t)\|_{L^\infty_\v(L^2_\h)}^2\,dt\leq & C_\d \exp\bigl(C_\d\|w_0^\h\|_{L^\infty_\v(L^2_\h)}^2(1+\|w_0^\h\|_{L^\infty_\v(L^2_\h)}^2)\bigr)\\
&\times\biggl(\f{\|\na_\h w_0^\h\|_{L^\infty_\v(L^2_\h)}^2\|w_0^\h\|_{L^\infty_\v((\dB^{-\d}_{2,\infty})_\h)}^{\f2\d} }{\|w_0^\h\|_{L^\infty_\v(L^2_\h)}^{\f2\d}}
+\|w_0^\h\|_{L^\infty_\v(L^2_\h)}^2\biggr).
\end{split}
\eeq
Yet by virtue of \eqref{S3eq3}, we have
\beno
\begin{split}
&\|w_0^\h\|_{L^\infty_\v(L^2_\h)}=\nu_\h^{-1}\|\buh_0\|_{L^\infty_\v(L^2_\h)}, \quad \|\na_\h w_0^\h\|_{L^\infty_\v(L^2_\h)}
=\nu_\h^{-\f12}\|\na_\h\buh_0\|_{L^\infty_\v(L^2_\h)}\\
&\|w_0^\h\|_{L^\infty_\v((\dB^{-\d}_{2,\infty})_\h)}=\nu_\h^{-1-\f\d2}\|u_0^\h\|_{L^\infty_\v((\dB^{-\d}_{2,\infty})_\h)}, \andf\\
&\int_0^\infty\|\na_\h w^\h(t)\|_{L^\infty_\v(L^2_\h)}^2\,dt=\nu_\h^{-1}\int_0^\infty\|\na_\h \buh(t)\|_{L^\infty_\v(L^2_\h)}^2\,dt,
\end{split}
\eeno
from which and \eqref{S3eq5}, we deduce \eqref{S3eq6}.
\end{proof}

\begin{prop}\label{S3prop2}
{\sl Under the assumptions of Lemma \ref{S3lem1}, for any $t>0,$ we have
\beq \label{S3eq7}
\begin{split}
\|\buh\|_{\wt{L}^\infty_t(B^{0,\f12})}+\sqrt{\nu_\h}\|\na_\h \buh&\|_{\wt{L}^2_t(B^{0,\f12})}+\sqrt{\nu_\v}\e\|\p_z \buh\|_{\wt{L}^2_t(B^{0,\f12})}\\
\lesssim&\|\buh_0\|_{B^{0,\f12}}\exp\Bigl(C\nu_\h^{-3}\|\buh_0\|_{L^\infty_\v(L^2_\h)}^2 A_{\nu_\h,\d}(\buh_0)\Bigr),
 \end{split}
\eeq where $A_{\nu_\h,\d}(\buh_0)$ is given by \eqref{S3eq6}.}
\end{prop}

\begin{proof} Let us denote
\beq \label{S3eq8}
g(t)\eqdefa \|\buh(t)\|_{L^\infty_\v(L^4_\h)}^4\andf f_\la(t)\eqdefa f(t)\exp\left(-\la \int_0^tg(t')\,dt'\right).
\eeq
Then by virtue of \eqref{S2eq2}, we write
\begin{equation*}
\partial_t\buh_\la+\la g(t) \buh_\la+\buh\cdot\na_\h\buh_\la-\nu_\h\Delta_\h \buh_\la-\nu_{\rm v}\e^2\partial_z^2 \buh_\la=-\na_\h p^\h_\la.
\end{equation*}
By applying $\D_\ell^\v$ to the above equation and then taking $L^2$ inner product of the resulting equation with
$\D_\ell^\v\buh_\la,$ we obtain
\beq \label{S3eq9}
\begin{split}
\f12\f{d}{dt}\|\D_\ell^\v\buh_\la(t)\|_{L^2}^2+&\la g(t)\|\D_\ell^\v\buh_\la(t)\|_{L^2}^2\\
+&\nu_\h\|\D_\ell^\v\na_\h\buh_\la\|_{L^2}^2
+\e^2\nu_\v\|\D_\ell^\v\p_z\buh_\la\|_{L^2}^2=\bigl(\D_\ell^\v(\buh\otimes\buh_\la) | \D_\ell^v\na_\h\buh_\la\bigr)_{L^2}.
\end{split}
\eeq
By applying Bony's decomposition \eqref{pd} to $\buh\otimes\buh_\la$ for the vertical variable, one has
\beno
\buh\otimes\buh_\la=2T^\v_{{\buh}}\buh_\la+R^\v(\buh,\buh_\la).
\eeno
Due to the support properties to the Fourier of the terms in $T^\v_{\buh}\buh_\la,$ we deduce
\beno
\begin{split}
\int_0^t&\bigl|\bigl(\D_\ell^\v (T^\v_{\buh}\buh_\la) | \D_\ell^v\na_\h\buh_\la\bigr)_{L^2}\bigr|\,dt'\\
\lesssim &\sum_{|\ell'-\ell|\leq 4}\int_0^t\|S_{\ell'-1}^\v\buh\|_{L^\infty_\v(L^4_\h)}\|\D_{\ell'}^\v\buh_\la\|_{L^2_\v(L^4_\h)}\|\D_\ell^\v\na_\h\buh_\la\|_{L^2}\,dt'\\
\lesssim &\sum_{|\ell'-\ell|\leq 4}\int_0^t\|\buh\|_{L^\infty_\v(L^4_\h)}\|\D_{\ell'}^\v\buh_\la\|_{L^2}^{\f12}\|\D_{\ell'}^\v\na_\h\buh_\la\|_{L^2}^{\f12}\|\D_\ell^\v\na_\h\buh_\la\|_{L^2}\,dt'.
\end{split}
\eeno
By applying H\"older's inequality and using Definition \ref{defpz}, we get
 \beno
\begin{split}
\int_0^t&\bigl|\bigl(\D_\ell^\v (T^\v_{\buh}\buh_\la) | \D_\ell^v\na_\h\buh_\la\bigr)_{L^2}\bigr|\,dt'\\
\lesssim &\sum_{|\ell'-\ell|\leq 4}\Bigl(\int_0^t\|\buh\|_{L^\infty_\v(L^4_\h)}^4\|\D_{\ell'}^\v\buh_\la\|_{L^2}^{2}\,dt\Bigr)^{\f14}\|\D_{\ell'}^\v\na_\h\buh_\la\|_{L^2_t(L^2)}^{\f12}
\|\D_\ell^\v\na_\h\buh_\la\|_{L^2_t(L^2)}\\
\lesssim &d_\ell^22^{-\ell}\|\buh_\la\|_{\wt{L}^2_{t,g}(B^{0,\f12})}^{\f12}\|\na_\h\buh_\la\|_{\wt{L}^2_{t}(B^{0,\f12})}^{\f32}.
\end{split}
\eeno
Along the same line, we have
\beno
\begin{split}
\int_0^t&\bigl|\bigl(\D_\ell^\v(R^\v({\buh},\buh_\la)) | \D_\ell^\v\na_\h\buh_\la\bigr)_{L^2}\bigr|\,dt'\\
\lesssim &\sum_{\ell'\geq \ell-3}\int_0^t\|\wt{\D}_{\ell'}^\v\buh\|_{L^\infty_\v(L^4_\h)}\|{\D}_{\ell'}^\v\buh_\la\|_{L^2_\v(L^4_h)}\|\D_\ell^\v\na_\h\buh_\la\|_{L^2}\,dt'\\
\lesssim &\sum_{\ell'\geq \ell-3}\Bigl(\int_0^t\|\buh\|_{L^\infty_\v(L^4_\h)}^4\|\D_{\ell'}^\v\buh_\la\|_{L^2}^{2}\,dt'\Bigr)^{\f14}\|\D_{\ell'}^\v\na_\h\buh_\la\|_{L^2_t(L^2)}^{\f12}
\|\D_\ell^\v\na_\h\buh_\la\|_{L^2_t(L^2)}\\
\lesssim &d_\ell 2^{-\ell}\|\buh_\la\|_{\wt{L}^2_{t,g}(B^{0,\f12})}^{\f12}\|\na_\h\buh\|_{\wt{L}^2_{t}(B^{0,\f12})}^{\f32}\sum_{\ell'\geq\ell-3}d_{\ell'}2^{-\f{\ell-\ell'}2},
\end{split}
\eeno
which implies
\beno
\int_0^t\bigl|\bigl(\D_\ell^\v(R^\v({\buh},\buh_\la)) | \D_\ell^\v\na_\h\buh_\la\bigr)_{L^2}\bigr|\,dt'
\lesssim d_\ell^22^{-\ell}\|\buh_\la\|_{\wt{L}^2_{t,g}(B^{0,\f12})}^{\f12}\|\na_\h\buh_\la\|_{\wt{L}^2_{t}(B^{0,\f12})}^{\f32}.
\eeno
As a result, it comes out
\beq\label{S3eq10}
\begin{split}
\int_0^t\bigl|\bigl(\D_\ell^\v({\buh}\otimes\buh_\la) | \D_\ell^\v\na_\h\buh_\la\bigr)_{L^2}\bigr|\,dt
\lesssim d_\ell^22^{-\ell}\|\buh_\la\|_{\wt{L}^2_{t,g}(B^{0,\f12})}^{\f12}\|\na_\h\buh_\la\|_{\wt{L}^2_{t}(B^{0,\f12})}^{\f32}.
\end{split}
\eeq
By integrating \eqref{S3eq9} over $[0,t]$ and then inserting \eqref{S3eq10} into the resulting inequality, we find
\beno
\begin{split}
\|\D_\ell^\v\buh_\la\|_{L^\infty_t(L^2)}^2+&\la\int_0^t g(t)\|\D_\ell^\v\buh_\la(t)\|_{L^2}^2\,dt
+\nu_\h\|\D_\ell^\v\na_\h\buh_\la\|_{L^2_T(L^2)}^2\\
+&\e^2\nu_\v\|\D_\ell^\v\p_z\buh_\la\|_{L^2_T(L^2)}^2\lesssim d_\ell^22^{-\ell}\|\buh_\la\|_{\wt{L}^2_{t,g}(\cB^{0,\f12})}^{\f12}\|\na_\h\buh_\la\|_{\wt{L}^2_{t}(\cB^{0,\f12})}^{\f32}.
\end{split}
\eeno
Multiplying the above inequality by $2^{\ell}$ and taking square root of the resulting inequality,
 and then summing up the resulting inequality over $\Z,$ we achieve
\beno
\begin{split}
&\|\buh_\la\|_{\wt{L}^\infty_t(B^{0,\f12})}+\sqrt{\la}\|\buh_\la\|_{\wt{L}^2_{t,g}(B^{0,\f12})}+\sqrt{\nu_\h}\|\na_\h\buh_\la\|_{\wt{L}^2_{T}(B^{0,\f12})}
+\sqrt{\nu_\v}\e\|\p_z\buh_\la\|_{\wt{L}^2_{t}(B^{0,\f12})}\\
&\leq \|\buh_0\|_{\cB^{0,\f12}}+C\|\buh_\la\|_{\wt{L}^2_{t,g}(B^{0,\f12})}^{\f14}\|\na_\h\buh_\la\|_{\wt{L}^2_{t}(B^{0,\f12})}^{\f34}\\
&\leq \|\buh_0\|_{\cB^{0,\f12}}+\f{\sqrt{\nu_\h}}2\|\na_\h\buh\|_{\wt{L}^2_{T}(B^{0,\f12})}+C\nu_\h^{-\f32}\|\buh_\la\|_{\wt{L}^2_{t,g}(B^{0,\f12})}.
\end{split}
\eeno
Taking $\la=C^2\nu_\h^{-3}$ in the above inequality leads to
\beno
\begin{split}
&\|\buh_\la\|_{\wt{L}^\infty_t(B^{0,\f12})}+\sqrt{\nu_\h}\|\na_\h\buh_\la\|_{\wt{L}^2_{t}(B^{0,\f12})}
+\sqrt{\nu_\v}\e\|\p_z\buh_\la\|_{\wt{L}^2_{t}(B^{0,\f12})}\leq 2\|\buh_0\|_{B^{0,\f12}},
\end{split}
\eeno
from which, and \eqref{S3eq8}, we infer
\beno
\begin{split}
&\Bigl(\|\buh\|_{\wt{L}^\infty_t(B^{0,\f12})}+\sqrt{\nu_\h}\|\na_\h\buh\|_{\wt{L}^2_{t}(B^{0,\f12})}
+\sqrt{\nu_\v}\e\|\p_z\buh\|_{\wt{L}^2_{t}(B^{0,\f12})}\Bigr)\exp\left(-\la \int_0^tg(t')\,dt'\right)\\
&\leq \|\buh_\la\|_{\wt{L}^\infty_T(B^{0,\f12})}+\sqrt{\nu_\h}\|\na_\h\buh_\la\|_{\wt{L}^2_{T}(B^{0,\f12})}
+\sqrt{\nu_\v}\e\|\p_z\buh_\la\|_{\wt{L}^2_{T}(B^{0,\f12})}\\
&\leq 2\|\buh_0\|_{B^{0,\f12}},
\end{split}
\eeno
Note that
\beno
\|\buh(t)\|_{L^\infty_\v(L^4_\h)}\leq C\|\buh(t)\|_{L^\infty_\v(L^2_\h)}^{\f12}\|\na_\h\buh(t)\|_{L^\infty_\v(L^2_\h)}^{\f12},
\eeno
we deduce that
\beno
\begin{split}
\Bigl(\|\buh\|_{\wt{L}^\infty_t(\cB^{0,\f12})}&+\sqrt{\nu_\h}\|\na_\h\buh\|_{\wt{L}^2_{t}(\cB^{0,\f12})}
+\sqrt{\nu_\v}\e\|\p_z\buh\|_{\wt{L}^2_{t}(\cB^{0,\f12})}\Bigr)\\
&\leq 2\|\buh_0\|_{\cB^{0,\f12}}\exp\Bigl(\f{C}{\nu_\h^3}\|\buh\|_{L^\infty_t(L^\infty_\v(L^2_\h))}^2\int_0^t \|\na_\h\buh(t')\|_{L^\infty_\v(L^2_\h)}^2\,dt'\Bigr).
\end{split}
\eeno
This together with \eqref{S3eq5a} and \eqref{S3eq6} ensures \eqref{S3eq7}.
\end{proof}

\begin{prop}\label{S3prop3}
{\sl Let $\buh$ be a smooth enough solution of \eqref{S2eq2}. Then for any $s'\in ]-1,1[$ and $s\in ]-1,0[,$ we have
\ben
\label{S3eq14} &&\|\buh\|_{L^\infty_t(\dH^{s',0})}^2+\nu_\h\|\na_\h\buh\|_{L^2_t(\dH^{s',0})}^2\leq \|\buh_0\|_{\dH^{s',0}}^2\exp\left(C\nu_\h^{-1}A_{\nu_\h,\d}(\buh_0)\right);\\
\label{S3eq15} &&\|\p_z\buh\|_{L^\infty_t(\dH^{s',0})}^2+\nu_\h\|\na_\h\p_z\buh\|_{L^2_t(\dH^{s',0})}^2\leq \|\p_z\buh_0\|_{\dH^{s',0}}^2\exp\left(C\nu_\h^{-1}A_{\nu_\h,\d}(\buh_0)\right);
\een
and \beq
\label{S3eq16} \begin{split}
\|\p_z^2\buh\|_{L^\infty_t(H^{s,0})}^2+&\nu_\h\|\na_\h\p_z^2\buh\|_{L^2_t(\dH^{s,0})}^2\\
\leq& \bigl(\|\p_z^2\buh_0\|_{\dH^{s,0}}^2+C\nu_\h^{-4}\|\p_z\buh_0\|_{\dH^{s,0}}^2\|\p_z\buh_0\|_{L^2}^4\bigr)\exp\left(C\nu_\h^{-1}A_{\nu_\h,\d}(\buh_0)\right),
\end{split}
\eeq where $A_{\nu_\h,\d}(\buh_0)$ is given by \eqref{S3eq6}.
}
\end{prop}

\begin{proof} In fact, \eqref{S3eq14} and \eqref{S3eq15} follows directly from    Lemma 4.2 of \cite{CZ5} and Lemma \ref{S3lem1}. To prove \eqref{S3eq16}, we get,
by first applying $\p_z^2$ to \eqref{S2eq2} and then taking $\dot{H}^s_\h$ inner product of the resulting equation with $\p_z^2\buh,$ that
\beq \label{S3eq17}
\begin{split}
\f12&\f{d}{dt}\|\p_z^2\buh(t,\cdot,z)\|_{\dH^s_\h}^2+\nu_\h\|\na_\h\p_z^2\buh(t,\cdot,z)\|_{\dH^s_\h}^2 +\nu_\v\e^2\|\p_z^2\buh(t,\cdot,z)\|_{\dH^s_\h}^2\\
&-\f{\nu_\v\e^2}2\p_z^2\|\p_z^2\buh(t,\cdot,z)\|_{\dH^s_\h}^2
=-\bigl(\buh\cdot\na_\h\p_z^2\buh(t,\cdot,z)  | \p_z^2\buh(t,\cdot,z)\bigr)_{\dH^s_\h}\\
&-2\bigl(\p_z\buh\cdot\na_\h\p_z\buh(t,\cdot,z)  | \p_z^2\buh(t,\cdot,z)\bigr)_{\dH^s_\h}-\bigl(\p_z^2\buh\cdot\na_\h\buh(t,\cdot,z)  | \p_z^2\buh(t,\cdot,z)\bigr)_{\dH^s_\h}.
\end{split}
\eeq
Applying Lemma 1.1 of \cite{ch92} yields
\beno
\begin{split}
\bigl|\bigl(\buh\cdot\na_\h\p_z^2\buh(t,\cdot,z)  | \p_z^2\buh(t,\cdot,z)\bigr)_{\dH^s_\h}\bigr|
\lesssim \|\na_\h\buh(t,\cdot,z)\|_{L^2_\h}\|\na_\h\p_z^2\buh(t,\cdot,z)\|_{\dH^{s}_\h}\|\p_z^2\buh(t,\cdot,z)\|_{\dH^s_\h}.
\end{split}
\eeno Due to $s\in ]-1,0[,$
the law of product in Sobolev spaces implies that
\beno
\begin{split}
\bigl|\bigl(\p_z^2\buh\cdot\na_\h\buh(t,\cdot,z)  | \p_z^2\buh(t,\cdot,z)\bigr)_{\dH^s_\h}\bigr|
\lesssim \|\na_\h\p_z^2\buh(t,\cdot,z)\|_{\dH^s_\h}\|\na_\h\buh(t,\cdot,z)\|_{L^2_\h}\|\p_z^2\buh(t,\cdot,z)\|_{\dH^s_\h}.
\end{split}
\eeno
Along the same line,  one has
\beno
\begin{split}
\bigl|\bigl(\p_z\buh&\cdot\na_\h\p_z\buh(t,\cdot,z)  | \p_z^2\buh(t,\cdot,z)\bigr)_{\dH^s_\h}\bigr|\\
=&\bigl|\bigl(\p_z\buh\otimes\p_z\buh(t,\cdot,z)  | \na_\h\p_z^2\buh(t,\cdot,z)\bigr)_{\dH^s_\h}\bigr|\\
\lesssim& \|\na_\h\p_z\buh(t,\cdot,z)\|_{\dH^s_\h}\|\p_z\buh(t,\cdot,z)\|_{L^2_\h}\|\na_\h\p_z^2\buh(t,\cdot,z)\|_{\dH^s_\h}.
\end{split}
\eeno
Inserting the above estimates into \eqref{S3eq17} and integrating the resulting equality with respect to $z$ gives
rise to
\beno
\begin{split}
\f12\f{d}{dt}\|\p_z^2\buh(t)\|_{\dH^{s,0}}^2+\nu_\h\|\na_\h\p_z^2\buh(t)\|_{\dH^{s,0}}^2
\leq C\Bigl(&\|\na_\h\buh\|_{L^\infty_\v(L^2_\h)}\|\na_\h\p_z^2\buh\|_{\dH^{s,0}}\|\p_z^2\buh\|_{\dH^{s,0}}\\
&+ \|\na_\h\p_z\buh\|_{L^\infty_\v(\dH^{s}_\h)}\|\p_z\buh\|_{L^2}\|\na_\h\p_z^2\buh\|_{\dH^{s,0}}\Bigr).
\end{split}
\eeno
Notice that
\beno
\|\na_\h\p_z\buh\|_{L^\infty_\v(\dH^{s}_\h)}\lesssim \|\na_\h\p_z\buh\|_{\dH^{s,0}}^{\f12}\|\na_\h\p_z^2\buh\|_{\dH^{s,0}}^{\f12},
\eeno
we get, by applying Young's inequality, that
\beno
\begin{split}
\f12&\f{d}{dt}\|\p_z^2\buh(t)\|_{\dH^{s,0}}^2+\nu_\h\|\na_\h\p_z^2\buh(t)\|_{\dH^{s,0}}^2
\leq \f{\nu_\h}2\|\na_\h\p_z^2\buh(t)\|_{\dH^{s,0}}^2\\
&\qquad+C{\nu_\h}^{-1}
\|\na_\h\buh\|_{L^\infty_\v(L^2_\h)}^2\|\p_z^2\buh\|_{\dH^{s,0}}^2+C\nu_\h^{-3}\|\na_\h\p_z\buh(t)\|_{\dH^{s,0}}^2\|\p_z\buh(t)\|_{L^2}^4.
\end{split}
\eeno
Applying Gronwall's inequality and using \eqref{S3eq15} leads to \eqref{S3eq16}.
\end{proof}

\begin{col}\label{S3col1}
{\sl Under the assumptions of Theorem \ref{thPZ9}, for any $\th\in [0,1/2[,$ we have
\beq \label{S3eq18}
\begin{split}
&\nu_\h^{\f{1-\th}2}\|\buh\|_{L^{\f2{1-\th}}_t(\cB^{\f12,\f12-\th})}\leq C\|\buh_0\|_{\dH^{-\f12,0}}^{\f12}\|\buh_0\|_{\dH^{\f12,0}}^{\th}\|\p_z\buh_0\|_{\dH^{-\f12,0}}^{\f12-\th}\exp\left(C\nu_\h^{-1}A_{\nu_\h,\d}(\buh_0)\right);\\
&\nu_\h^{\f12}\|\p_z\buh\|_{L^2_t(\cB^{\f12,\f12})}\leq C\|\p_z\buh_0\|_{\dH^{-\f12,0}}^{\f12}\\
&\qquad\times \bigl(\|\p_z^2\buh_0\|_{\dH^{-\f12,0}}
+\nu_\h^{-4}\|\p_z\buh_0\|_{\dH^{-\f12,0}}^2 \|\p_z\buh_0\|_{L^2}^4\bigr)
^{\f12}\exp\left(C\nu_\h^{-1}A_{\nu_\h,\d}(\buh_0)\right).
\end{split}
\eeq
}
\end{col}
\begin{proof} Indeed it follows from interpolation inequality in Besov spaces that
\beno
\begin{split}
&\|\buh\|_{L^{\f2{1-\th}}_t(\cB^{\f12,\f12-\th})}\lesssim \|\buh\|_{L^\infty_t(H^{\f12,0})}^{\th}\|\buh\|_{L^2_t(H^{\f12,0})}^{\f12}\|\p_z\buh\|_{L^2_t(H^{\f12,0})}^{\f12-\th};\\
&\|\p_z\buh\|_{L^2_t(\cB^{\f12,\f12})}
\lesssim\|\p_z\buh\|_{L^2_t(H^{\f12,0})}^{\f12}\|\p_z^2\buh\|_{L^2_t(H^{\f12,0})}^{\f12},
\end{split}
\eeno
which together with Proposition \ref{S3prop3} ensures \eqref{S3eq18}.
\end{proof}

 \setcounter{equation}{0}
\section{ The proof of Theorem \ref{thPZ9}}\label{Sect5}

The goal of this section is to present the proof of Theorem \ref{thPZ9}.

\begin{proof}[Proof of Theorem \ref{thPZ9}] For simplicity, we just present the {\it priori} estimates for smooth enough solutions of \eqref{S2eq4}.
 Let $\buh$ be a smooth enough solution of \eqref{S2eq2}.
 Let $\th\in [0,1/2[,$  $\la_i, i=1,2,3,$ and $\na_\e\eqdefa (\na_\h, \e \p_3),$ we denote
\beq \label{S4eq1}
\begin{split}
&f_1(t)\eqdefa \|\buh(t)\|_{B^{0,\f12}}^2\|\na_\h\buh(t)\|_{B^{0,\f12}}^2, \quad f_2(t)\eqdefa \|\na_\e \buh(t)\|_{B^{0,\f12}}^2,\\
& f_3(t)\eqdefa\|\buh(t)\|_{\cB^{\f12,\f12-\th}}^{\f2{1-\th}} +\|\p_z\buh(t)\|_{\cB^{\f12,\f12}}^2\andf\\
& R_\la(t)\eqdefa R(t)\exp\Bigl(-\sum_{i=1}^3\la_i\int_0^tf_i(t')\,dt'\Bigr).
\end{split}
\eeq
And similar notations for $\pi_\la$ and $F_\la.$
Then it follows from \eqref{S2eq4} that
\begin{equation}\label{S4eq2}
\begin{split}
\partial_tR_\la+\Bigl(\sum_{i=1}^3\la_i f_i(t)\Bigr)R_\la+&u\cdot\na R_\la+R_\la\cdot\na(v_L+[\buh]_\e)\\
&-\nu_\h\Delta_\h R_\la-\nu_{\rm v}\partial_3^2 R_\la+\na \pi_\la=F_\la,
\end{split}
\end{equation}
where $F_\la=(F_\la^\h,F^3_\la)$ with $(F^\h,F^3)$ being given by \eqref{S2eq4}.

Applying the operator $\D_\ell^\v$ to \eqref{S4eq2} and taking $L^2$ inner product of the resulting equation with
$\D_\ell^\v R_\la$ yields
\beq \label{S4eq3}
\begin{split}
\f12&\f{d}{dt}\|\D_\ell^\v R_\la(t)\|_{L^2}^2+\sum_{i=1}^3\la_i f_i(t)\|\D_\ell^\v R_\la(t)\|_{L^2}^2 +\nu_\h\|\na_\h\D_\ell^\v R_\la\|_{L^2}^2+\nu_{\rm v}\|\partial_3\D_\ell^\v R_\la\|_{L^2}^2\\
&=-\bigl(\D_\ell^\v(u\cdot\na R_\la+R_\la\cdot\na(v_L+[\buh]_\e)) | \D_\ell^\v R_\la\bigr)_{L^2}+\bigl(\D_\ell^\v F_\la |  \D_\ell^\v R_\la\bigr)_{L^2}.
\end{split}
\eeq

The estimate of the above terms relies on the following lemmas:

\begin{lem}\label{S4lem2}
{\sl  There holds
\beno
\int_0^t\bigl|\bigl(\D_\ell^\v(a_\otimes [\buh]_\e) | \D_\ell^\v b\bigr)_{L^2}\bigr|\,dt\lesssim  d_\ell^2 2^{-{\ell}}
\|a\|_{\wt{L}^2_{t,f_1}(B^{0,\f12})}^{\f12}\|\na_\h a\|_{\wt{L}^2_{t}(B^{0,\f12})}^{\f12}\|b\|_{\wt{L}^2_{t}(B^{0,\f12})}.
\eeno}
\end{lem}

\begin{lem}\label{S4lem3}
{\sl There holds
\beno
\int_0^t\bigl|\bigl(\D_\ell^\v([\buh]_\e\cdot\na_\h v_L) | \D_\ell^\v  a\bigr)_{L^2}\bigr|\,dt'\lesssim d_\ell^2 2^{-{\ell}}\|\na_\h v_L\|_{\wt{L}^2_t(B^{0,\f12})}
\|a\|_{\wt{L}^2_{t,f_1}(B^{0,\f12})}^{\f12}\|\na_\h a\|_{\wt{L}^2_{t}(B^{0,\f12})}^{\f12}.
\eeno}
\end{lem}

\begin{lem}\label{S4lem5}
{\sl There holds
\beno
\begin{split}
\int_0^te^{-\la_2\int_0^{t'}f_2(\tau)\,d\tau}\bigl|\bigl(\D_\ell^\v(v_L\cdot[\na_\e \buh]_\e) | \D_\ell^\v & a\bigr)_{L^2}\bigr|\,dt'\lesssim  d_\ell^2 2^{-{\ell}}\la_2^{-\f14}\|v_L\|_{\wt{L}^\infty_t(\cB^{0,\f12})}^{\f12}
 \\
&\times \|\na_\h v_L\|_{\wt{L}^2_t(B^{0,\f12})}^{\f12}\| a\|_{\wt{L}^2_{t,f_2}(B^{0,\f12})}^{\f12}
\| \na_\h a\|_{\wt{L}^2_{t}(B^{0,\f12})}^{\f12}. \end{split}
\eeno}
\end{lem}

\begin{lem}\label{S4lem6}
{\sl Let $p^\h$ be given by \eqref{S2eq5}. Then  for any $\th\in [0,1/2[,$ we have
\beno
\begin{split}
\int_0^t\exp\Bigl(-\la_3\int_0^{t'}f_3(\tau)\,d\tau\Bigr)&\bigl|\bigl(\D_\ell^\v\p_3[p^\h]_\e | \D_\ell^\v a\bigr)_{L^2}\bigr|\,dt'\\
&\lesssim
 d_\ell^22^{-\ell}\e^{1-\th}\la_3^{-\f12}\|a \|_{\wt{L}^2_{t,f_3}(B^{0,\f12})}^{1-\th}\|\p_3a\|_{\wt{L}^2_t(B^{0,\f12})}^\th.
\end{split}
\eeno}
\end{lem}

We shall postpone the proof of the above lemmas in the Appendix.

Let us admit the above Lemmas for the time being and continue to handle the terms in the second line of \eqref{S4eq3}.
We first get, by using integration by parts, that
\beno
\begin{split}
&\int_0^t\bigl(\D_\ell^\v(R\cdot\na R_\la) | \D_\ell^\v R_\la\bigr)_{L^2}\,dt'=-\int_0^t\bigl(\D_\ell^\v(R \otimes R_\la) | \D_\ell^\v \na R_\la\bigr)_{L^2}\,dt',\\
&\int_0^t\bigl(\D_\ell^\v(v_L\cdot\na v_L)_\la | \D_\ell^\v R_\la\bigr)_{L^2}\,dt'=\int_0^t\bigl(\D_\ell^\v(v_L\otimes v_L)_\la | \D_\ell^\v \na R_\la\bigr)_{L^2}\,dt'.
\end{split}
\eeno
Then applying the law of product, Lemma \ref{S3lem0}, gives
\beno
\begin{split}
\int_0^t\bigl|\bigl(\D_\ell^\v(R\cdot\na R_\la) | \D_\ell^\v R_\la\bigr)_{L^2}\bigr|\,dt'\leq&
\|\D_\ell^\v(R \otimes R_\la)\|_{L^{2}_t(L^2)}\|\D_\ell^\v \na R_\la\|_{L^2_t(L^2)}\\
\lesssim & d_\ell^2 2^{-\ell}\|R \otimes R_\la\|_{\wt{L}^2_t(B^{0,\f12})}\|\na R_\la\|_{\wt{L}^2_t(B^{0,\f12})}\\
\lesssim
&d_\ell^2 2^{-{\ell}}\|R\|_{\wt{L}^\infty_t(B^{0,\f12})}\|\na_\h R_\la\|_{\wt{L}^2_t(B^{0,\f12})}
 \|\na R_\la\|_{\wt{L}^2_t(B^{0,\f12})},
\end{split}
\eeno
and
\beno
\begin{split}
\int_0^t\bigl|\bigl(\D_\ell^\v(v_L\cdot\na v_L) | \D_\ell^\v R_\la\bigr)_{L^2}\bigr|\,dt'
\lesssim & d_\ell^22^{-\ell} \|v_L\|_{\wt{L}^\infty_t(B^{0,\f12})}\|\na_\h v_L\|_{\wt{L}^2_t(B^{0,\f12})}\|\na R_\la\|_{\wt{L}^2_t(B^{0,\f12})}.
\end{split}
\eeno
Along the same line, by using integrating by parts, one has
\beno
\begin{split}
&\int_0^t\bigl(\D_\ell^\v(v_L\cdot\na R_\la+R_\la\cdot\na v_L) | \D_\ell^\v R_\la\bigr)_{L^2}\,dt'\\
&=-\int_0^t\bigl(\D_\ell^\v(v_L\otimes R_\la+R_\la\otimes v_L) | \D_\ell^\v \na R_\la\bigr)_{L^2}\,dt'.
\end{split}
\eeno
It follows from the law of product in anisotropic Besov space, Lemma \ref{S3lem0}, that
\beno
\begin{split}
\|v_L\otimes R_\la\|_{\wt{L}^2_t(B^{0,\f12})}
\lesssim &\|v_L\|_{\wt{L}^\infty_t(B^{0,\f12})}^{\f12}\|\na_\h v_L\|_{\wt{L}^2_t(B^{0,\f12})}^{\f12}
\|R_\la\|_{\wt{L}^\infty_t(B^{0,\f12})}^{\f12}\|\na_\h R_\la\|_{\wt{L}^2_t(B^{0,\f12})}^{\f12},
\end{split}
\eeno
which implies
\beno
\begin{split}
&\int_0^t\bigl|\bigl(\D_\ell^\v(v_L\cdot\na R_\la+R_\la\cdot\na v_L) | \D_\ell^\v R_\la\bigr)_{L^2}\bigr|\,dt'\\
&\lesssim d_\ell^22^{-\ell} \|v_L\|_{\wt{L}^\infty_t(B^{0,\f12})}^{\f12}\|\na_\h v_L\|_{\wt{L}^2_t(B^{0,\f12})}^{\f12}
\|R_\la\|_{\wt{L}^\infty_t(B^{0,\f12})}^{\f12}\|\na_\h R_\la\|_{\wt{L}^2_t(B^{0,\f12})}^{\f12}\|\na R_\la\|_{\wt{L}^2_t(B^{0,\f12})}.
\end{split}
\eeno

Whereas we get, by using integrating by parts, that
\beno
\begin{split}
&\int_0^t\bigl(\D_\ell^\v([\buh]_\e\cdot\na_\h R_\la+R_\la\cdot\na [\buh]_\e) | \D_\ell^\v R_\la\bigr)_{L^2}\,dt'\\
&=-\int_0^t\bigl(\D_\ell^\v([\buh]_\e\otimes R_\la) | \D_\ell^\v \na_\h R_\la\bigr)_{L^2}\,dt'
-\int_0^t\bigl(\D_\ell^\v(R_\la \otimes [\buh]_\e) | \D_\ell^\v \na R_\la\bigr)_{L^2}\,dt'.
\end{split}
\eeno
Applying Lemma \ref{S4lem2} leads to
\beno
\begin{split}
\int_0^t\bigl|\bigl(\D_\ell^\v([\buh]_\e\cdot\na_\h R_\la&+R_\la\cdot\na [\buh]_\e) | \D_\ell^\v R_\la\bigr)_{L^2}\bigr|\,dt'\\
&\lesssim d_\ell^2 2^{-{\ell}}\|R\|_{\wt{L}^2_{t,f_1}(B^{0,\f12})}^{\f12}
 \|\na_\h R_\la\|_{\wt{L}^2_t(B^{0,\f12})}^{\f12}\|\na R_\la\|_{\wt{L}^2_t(B^{0,\f12})}.
\end{split}
\eeno
Applying Lemma \ref{S4lem3} gives
\beno
\int_0^t\bigl|\bigl(\D_\ell^\v([\buh]_\e\cdot\na_\h v_L) | \D_\ell^\v  R_\la\bigr)_{L^2}\bigr|\,dt'\lesssim d_\ell^2 2^{-{\ell}}\|\na_\h v_L\|_{\wt{L}^2_T(B^{0,\f12})}
\|R_\la\|_{\wt{L}^2_{t,f_1}(B^{0,\f12})}^{\f12}\|\na_\h R_\la\|_{\wt{L}^2_{t}(B^{0,\f12})}^{\f12}.
\eeno
Applying Lemma \ref{S4lem5} yields
\beno
\begin{split}
\int_0^t\bigl|\bigl(\D_\ell^\v(v_L\cdot[\na_\e \buh]_\e)_\la | \D_\ell^\v & R_\la\bigr)_{L^2}\bigr|\,dt'\lesssim  d_\ell^2 2^{-{\ell}}\la_2^{-\f14}\|v_L\|_{\wt{L}^\infty_t(B^{0,\f12})}^{\f12}
 \\
&\times \|\na_\h v_L\|_{\wt{L}^2_t(B^{0,\f12})}^{\f12}\| R_\la\|_{\wt{L}^2_{t,f_2}(B^{0,\f12})}^{\f12}
\| \na_\h R_\la\|_{\wt{L}^2_{t}(B^{0,\f12})}^{\f12}. \end{split}
\eeno
Finally for any $\th\in [0,\frac12[,$ we get, by applying Lemma \ref{S4lem6}, that
\beno
\begin{split}
\int_0^t&\bigl|\bigl(\D_\ell^\v(\p_3[p^\h]_\e)_\la | \D_\ell^\v R^3\bigr)_{L^2}\bigr|\,dt'\lesssim
 d_\ell^22^{-\ell}\e^{1-\th}\la_3^{-\f12}\|R^3\|_{\wt{L}^2_{t,f_3}(B^{0,\f12})}^{1-\th}\|\p_3R^3\|_{\wt{L}^2_t(B^{0,\f12})}^\th.
\end{split}
\eeno
Let us denote
\beno
\cA_t\eqdefa \|v_L\|_{\wt{L}^\infty_t(B^{0,\f12})}
  \|\na_\h v_L\|_{\wt{L}^2_t(B^{0,\f12})}. \eeno
  Then it follows from Proposition \ref{S3prop1} that
\beq \label{S5eq2}
\cA_t\lesssim \left(\nu_\h\nu_\v\right)^{-\f14}\cA_0\with \cA_0\eqdefa\|v_0\|_{\cB^{\f12,0}}\bigl(\|v_0\|_{B^{0,\f12}}+\|v_0\|_{\cB^{\f12,0}}\bigr).
\eeq
By integrating \eqref{S4eq3} over $[0,t]$ and inserting the above estimates into the resulting inequality, and then  we take the square root of
resulting inequality to achieve
\beno
\begin{split}
\|\D_\ell^\v R_\la\|_{L^\infty_t(L^2)}&+\sum_{i=1}^3\la_i^{\f12}\Bigl( \int_0^tf_i(t')\|\D_\ell^\v R_\la(t')\|_{L^2}^2\,dt'\Bigr)^{\f12}\\
 &\qquad\qquad\ \ +\nu_\h^{\f12}\|\na_\h\D_\ell^\v R_\la\|_{L^2_t(L^2)}
+\nu_{\rm v}^{\f12}\|\partial_3\D_\ell^\v R_\la\|_{L^2_t(L^2)}\\
\lesssim  d_{\ell}2^{-\f{\ell}2}\biggl(&\|\na R_\la\|_{\wt{L}^2_t(B^{0,\f12})}^{\f12}\Bigl[\cA^{\f12}_t+\|R\|_{\wt{L}^\infty_t(B^{0,\f12})}^{\f12}\|\na_\h R_\la\|_{\wt{L}^2_t(B^{0,\f12})}^{\f12}\\
&\quad +\bigl(\|R_\la\|_{\wt{L}^2_{t,f_1}(B^{0,\f12})}^{\f14}+\cA^{\f14}_t\|R\|_{\wt{L}^\infty_t(B^{0,\f12})}^{\f14}
 \bigr)
 \|\na_\h R_\la\|_{\wt{L}^2_t(B^{0,\f12})}^{\f14}\Bigr]\\
 &+\Bigl[\la_2^{-\f18} \cA^{\f14}_t\| R_\la\|_{\wt{L}^2_{t,f_2}(B^{0,\f12})}^{\f14}
 +\|\na_\h v_L\|_{\wt{L}^2_T(B^{0,\f12})}^{\f12}
\|R_\la\|_{\wt{L}^2_{t,f_1}(B^{0,\f12})}^{\f14}\Bigr]\\
&\quad\times \|\na_\h R_\la\|_{\wt{L}^2_{t}(B^{0,\f12})}^{\f14}+\e^{\f{1-\th}2} \la_3^{-\f14}\|R^3\|_{\wt{L}^2_{t,f_3}(B^{0,\f12})}^{\f{1-\th}2}\|\p_3R^3\|_{\wt{L}^2_t(B^{0,\f12})}^{\f{\th}2}\biggr).
\end{split}
\eeno
Multiplying $2^{\f{\ell}2}$ to the above inequality and summing  up the resulting inequality for $\ell\in\Z$ leads to
\beq\label{S5eq1}
\begin{split}
\|R_\la&\|_{\wt{L}^\infty_t(B^{0,\f12})}+\sum_{i=1}^3{\la_i}^{\f12}\|R_\la\|_{\wt{L}^2_{t,f_i}(B^{0,\f12})}+\nu_\h^{\f12}\|\na_\h R_\la\|_{\wt{L}^2_t(B^{0,\f12})}
+\nu_\v^{\f12}\|\p_3 R_\la\|_{\wt{L}^2_t(B^{0,\f12})}\\
\leq & C\biggl(\|\na R_\la\|_{\wt{L}^2_t(B^{0,\f12})}^{\f12}\Bigl[\cA^{\f12}_t+\|R\|_{\wt{L}^\infty_t(B^{0,\f12})}^{\f12}\|\na_\h R_\la\|_{\wt{L}^2_t(B^{0,\f12})}^{\f12}\\
 &\qquad+\bigl(\cA^{\f14}_t\|R\|_{\wt{L}^\infty_t(B^{0,\f12})}^{\f14}+\|R_\la\|_{\wt{L}^2_{t,f_1}(B^{0,\f12})}^{\f14}\bigr)
 \|\na_\h R_\la\|_{\wt{L}^2_t(B^{0,\f12})}^{\f14}\Bigr]\\
 &\quad+\Bigl[\la_2^{-\f18} \cA^{\f14}_t\| R_\la\|_{\wt{L}^2_{t,f_2}(B^{0,\f12})}^{\f14}
 +\|\na_\h v_L\|_{\wt{L}^2_T(B^{0,\f12})}^{\f12}
\|R_\la\|_{\wt{L}^2_{t,f_1}(B^{0,\f12})}^{\f14}\Bigr]\\
&\qquad\times \|\na_\h R_\la\|_{\wt{L}^2_{t}(B^{0,\f12})}^{\f14}+\e^{\f{1-\th}2} \la_3^{-\f14}\|R^3\|_{\wt{L}^2_{t,f_3}(B^{0,\f12})}^{\f{1-\th}2}\|\p_3R^3\|_{\wt{L}^2_t(B^{0,\f12})}^{\f{\th}2}\biggr).
\end{split}
\eeq
Let us assume that $\nu_\v\geq \nu_\h.$ Then by applying Young's inequality, we obtain
\beno
\begin{split}
&C\cA^{\f12}_t\|\na R_\la\|_{\wt{L}^2_t(B^{0,\f12})}^{\f12}\leq C\nu_\h^{-\f12}\cA_t+\f{\nu_\h^{\f12}}{100}\|\na R_\la\|_{\wt{L}^2_t(B^{0,\f12})};\\
 &C\cA^{\f14}_t\|R\|_{\wt{L}^\infty_{t}(B^{0,\f12})}^{\f14}
 \|\na R_\la\|_{\wt{L}^2_t(B^{0,\f12})}^{\f34}\leq C\cA_t \nu_\h^{-\f32}\|R\|_{\wt{L}^\infty_{t}(B^{0,\f12})}
 +\f{\nu_\h^{\f12}}{100}\|\na R_\la\|_{\wt{L}^2_t(B^{0,\f12})};\\
&C\|R\|_{\wt{L}^2_{t,f_1}(B^{0,\f12})}^{\f14}
 \|\na R_\la\|_{\wt{L}^2_t(B^{0,\f12})}^{\f34}\leq \f{C}2\nu_\h^{-\f32}\|R\|_{\wt{L}^2_{t,f_1}(B^{0,\f12})}
 +\f{\nu_\h^{\f12}}{100}\|\na R_\la\|_{\wt{L}^2_t(B^{0,\f12})};\\
\end{split}
 \eeno
 and
 \beno
 \begin{split}
C\la_2^{-\f18} \cA^{\f14}_t&\| R_\la\|_{\wt{L}^2_{t,f_2}(B^{0,\f12})}^{\f14}
  \|\na_\h R_\la\|_{\wt{L}^2_{t}(B^{0,\f12})}^{\f14}\\
  &\leq  C\cA^{\f12}_t\la_2^{-\f14}
  +C\nu_\h^{-\f12}\| R_\la\|_{\wt{L}^2_{t,f_2}(B^{0,\f12})}+
\f{\nu_\h^{\f12}}{100}\|\na R_\la\|_{\wt{L}^2_t(B^{0,\f12})};
\end{split}
 \eeno
 and
 \beno
 \begin{split}
 \|\na_\h v_L&\|_{\wt{L}^2_T(B^{0,\f12})}^{\f12}
\|R_\la\|_{\wt{L}^2_{t,f_1}(B^{0,\f12})}^{\f14}\|\na_\h R_\la\|_{\wt{L}^2_{t}(B^{0,\f12})}^{\f14}\\
&\leq C\nu_\h^{\f12}\|\na_\h v_L\|_{\wt{L}^2_T(B^{0,\f12})}
+\f{C}2\nu_\h^{-\f32}\| R_\la\|_{\wt{L}^2_{t,f_1}(B^{0,\f12})}+
\f{\nu_\h^{\f12}}{100}\|\na R_\la\|_{\wt{L}^2_t(B^{0,\f12})};
\end{split}
 \eeno
 and
 \beno
 \begin{split}
 C\e^{\f{1-\th}2} \la_3^{-\f14}&\|R^3\|_{\wt{L}^2_{t,f_3}(\cB^{0,\f12})}^{\f{1-\th}2}\|\p_3R^3\|_{\wt{L}^2_t(\cB^{0,\f12})}^{\f{\th}2}\\
 \leq&
 C\e^{1-\th}\nu_\h^{-\f{1-\th}2}\nu_\v^{-\f{\th}2}\la_3^{-\f12}+\nu_\h^{\f12}\|R^3\|_{\wt{L}^2_{t,f_3}(\cB^{0,\f12})}+\f{\nu_\v^{\f12}}{100}\|\p_3 R_\la\|_{\wt{L}^2_t(B^{0,\f12})}.
 \end{split}
 \eeno
 Inserting the above estimates into \eqref{S5eq1} leads to
 \beq \label{S5eq3q}
\begin{split}
\|R_\la&\|_{\wt{L}^\infty_t(B^{0,\f12})}+\sum_{i=1}^3{\la_i}^{\f12}\|R\|_{\wt{L}^2_{t,f_i}(B^{0,\f12})}+\nu_\h^{\f12}\|\na_\h R_\la\|_{\wt{L}^2_t(B^{0,\f12})}
+\nu_\v^{\f12}\|\p_3 R_\la\|_{\wt{L}^2_t(B^{0,\f12})}\\
\leq & C\Bigl(\cA_t\nu_\h^{-\f12}+\cA^{\f12}_t\la_2^{-\f14}+\nu_\h^{\f12}\|\na_\h v_L\|_{\wt{L}^2_T(B^{0,\f12})}+\e^{1-\th}\nu_\h^{-\f{1-\th}2}\nu_\v^{-\f{\th}2}\la_3^{-\f12}\Bigr)\\
&+C\cA_t\nu_\h^{-\f32}\|R\|_{\wt{L}^\infty_t(B^{0,\f12})}+
\bigl(\f{\nu_h^{\f12}}4+C\|R\|_{\wt{L}^\infty_t(B^{0,\f12})}^{\f12}\bigr)\|\na R_\la\|_{\wt{L}^2_t(B^{0,\f12})}\\
&+C\bigl(\nu_\h^{-\f32}\|R\|_{\wt{L}^2_{t,f_1}(B^{0,\f12})}+\nu_\h^{-\f12}\|R\|_{\wt{L}^2_{t,f_2}(B^{0,\f12})}\bigr)+\nu_\h^{\f12}\|R^3\|_{\wt{L}^2_{t,f_3}(B^{0,\f12})}
\end{split}
\eeq

In view of \eqref{S5eq3q}, we deduce from standard theory of Navier-Stokes system that \eqref{S2eq4} has a unique solution
$R\in C([0,T^\ast[; B^{0,\frac12})$ with $\na R\in L^2(]0,T^\ast[; B^{0,\f12})$ for some maximal existing time $T^\ast$. We
are going to prove that $T^\ast=\infty$  under the smallness
condition \eqref{S1eq1}.  Otherwise, we
denote
\beq \label{S5eq3}
T^\star\eqdefa \sup\bigl\{\ T<T^\ast:\ \|R\|_{\wt{L}^\infty_t(B^{0,\f12})}\leq \f{\nu_\h}{16C^2}\ \ \bigr\},
\eeq
and we take
\beq \label{S5eq3a}
\la_1=C^2\nu_\h^{-3}, \quad \la_2=C^2\nu_\h^{-1} \andf \la_3=\nu_\h. \eeq
Then for $t\leq T^\star,$ we deduce from \eqref{S5eq2}, \eqref{S5eq3q} and  \eqref{S5eq3} that
\beno
\begin{split}
\|R_\la\|_{\wt{L}^\infty_t(B^{0,\f12})}+&\nu_\h^{\f12}\|\na_\h R_\la\|_{\wt{L}^2_t(B^{0,\f12})}
+\nu_\v^{\f12}\|\p_3 R_\la\|_{\wt{L}^2_t(B^{0,\f12})}\\
&\qquad\qquad\leq  C\Bigl(\cA_0^{\f12}\bigl(1+\cA_0^{\f12}\bigr)\bigl(\nu_\h^{-\f78}+\nu_\h^{\f18}\bigr)\nu_\v^{-\f18}+\e^{1-\th}\nu_\h^{-1+\f\th2}\nu_\v^{-\f{\th}2}\Bigr),
\end{split}
\eeno
which together with \eqref{S4eq1} implies that
\beno
\begin{split}
&\|R\|_{\wt{L}^\infty_t(B^{0,\f12})}+\nu_\h^{\f12}\|\na_\h R\|_{\wt{L}^2_t(B^{0,\f12})}
+\nu_\v^{\f12}\|\p_3 R\|_{\wt{L}^2_t(B^{0,\f12})}\\
&\qquad\leq  C\Bigl(\cA_0^{\f12}\bigl(1+\cA_0^{\f12}\bigr)\bigl(\nu_\h^{-\f78}+\nu_\h^{\f18}\bigr)\nu_\v^{-\f18}+\e^{1-\th}\nu_\h^{-1+\f\th2}\nu_\v^{-\f{\th}2}\Bigr)\\
&\qquad\quad\times
\exp\Bigl(C\nu_\h^{-1}\int_0^t \bigl(1+ \nu_\h^{-2}\|\buh(t')\|_{B^{0,\f12}}^2\bigr)\|\na_\e\buh(t')\|_{B^{0,\f12}}^2\,dt'\\
&\qquad\qquad\qquad\ +C\nu_\h\int_0^t \bigl(\|\buh(t')\|_{\cB^{\f12,\f12-\th}}^{\f2{1-\th}} +\|\p_z\buh(t)\|_{\cB^{\f12,\f12}}^2\bigr)\,dt'\Bigr),
\end{split}
\eeno
from which and Proposition \ref{S3prop2} and Corollary \ref{S3col1}, we deduce that for $t\leq T^\star,$
\beq \label{S5eq4}
\begin{split}
\|R\|_{\wt{L}^\infty_t(B^{0,\f12})}+&\nu_\h^{\f12}\|\na_\h R\|_{\wt{L}^2_t(B^{0,\f12})}
+\nu_\v^{\f12}\|\p_3 R\|_{\wt{L}^2_t(B^{0,\f12})}\\
\leq  & C\Bigl(\cA_0^{\f12}\bigl(1+\cA_0^{\f12}\bigr)\bigl(\nu_\h^{-\f78}+\nu_\h^{\f18}\bigr)\nu_\v^{-\f18}+\e^{1-\th}\nu_\h^{-1+\f\th2}\nu_\v^{-\f{\th}2}\Bigr)\cA_{\d,\nu_h}(\buh_0)\with\\
\cA_{\d,\nu_h}(\buh_0)\eqdefa &
\exp\biggl(C\nu_\h^{-2}\|\buh_0\|_{B^{0,\f12}}^2\bigl(1+\nu_\h^{-2}\|\buh_0\|_{B^{0,\f12}}^2\bigr)\exp\Bigl(C\nu_\h^{-3}\|\buh_0\|_{L^\infty_\v(L^2_\h)}^2 A_{\nu_\h,\d}(\buh_0)\Bigr)\\
&\qquad+C\exp\left(C\nu_\h^{-1}A_{\nu_\h,\d}(\buh_0)\right)\Bigl(\|\buh_0\|_{\dH^{-\f12,0}}^{\f1{1-\th}}\|\buh_0\|_{\dH^{\f12,0}}^{\f{2\th}{1-\th}}\||\p_z\buh_0\|_{\dH^{-\f12,0}}^{\f{1-2\th}{1-\th}}
\\
&\qquad\quad +\|\p_z\buh_0\|_{\dH^{-\f12,0}} \bigl(\|\p_z^2\buh_0\|_{\dH^{-\f12,0}}
+\nu_\h^{-4}\|\p_z\buh_0\|_{\dH^{-\f12,0}}^2 \|\p_z\buh_0\|_{L^2}^4\bigr)\Bigr)
\biggr),
\end{split}
\eeq where $A_{\nu_\h,\d}(\buh_0)$ is given by \eqref{S3eq6}.

Then under the smallness condition \eqref{S1eq1}, we obtain
\beno
\|R\|_{\wt{L}^\infty_{T^\star}(B^{0,\f12})}+\nu_\h^{\f12}\|\na_\h R\|_{\wt{L}^2_{T^\star}(B^{0,\f12})}
+\nu_\v^{\f12}\|\p_3 R\|_{\wt{L}^2_{T^\star}(B^{0,\f12})}
\leq \f{\nu_\h}{32C^2}.
\eeno
This contradicts with \eqref{S5eq3}, which  in torn shows that $T^\ast=T^\star=\infty.$
This completes the proof of Theorem \ref{thPZ9}.
\end{proof}

\setcounter{equation}{0}

\section{The case of fluid evolving between two parallel plans }

The goal of this section is to investigate the global well-posedness of the anisotropic Navier-Stokes system with only vertical viscosity,  \eqref{S1eq4}.
To do it, let us first present the following lemmas:

\begin{lem}\label{S6lem1}
{\sl Let $\Om=\R^2\times]0,1[,$ and $f$ satisfy $f|_\Om=0$ and $\p_3f\in \cB_h^1(\Om)$ (see Definition \ref{def4}). Then one has
\beno
\|f\|_{L^\infty}\leq C\|\p_3f\|_{\cB^1_\h}.
\eeno}
\end{lem}

\begin{proof} We first get, by applying Lemma \ref{lemBern}, that
\beno
\begin{split}
\|f\|_{L^\infty}\leq &\sum_{k\in\Z}\|\D_k^\h f\|_{L^\infty}
\lesssim \sum_{k\in\Z}2^k\|\D_k^\h f\|_{L^\infty_\v(L^2_\h)}\\
\lesssim &\sum_{k\in\Z}2^k\|\D_k^\h f\|_{L^2}^{\f12}\|\D_k^\h \p_3 f\|_{L^2}^{\f12}.
\end{split}
\eeno
Recalling that Poincar'e inequality holds in the strip $\Om$ with Dirichlet boundary condition
\beq \label{S6eq1}
\|f\|_{L^2(\Omega)}\leq C\|\partial_3 f\|_{L^2(\Omega)}.\eeq
So that we obtain
\beno
\begin{split}
\|f\|_{L^\infty}
\lesssim \sum_{k\in\Z}2^k\|\D_k^\h \p_3 f\|_{L^2}\lesssim \|\p_3f\|_{\cB^1_\h}.
\end{split}
\eeno
This completes the proof of the lemma.
\end{proof}

\begin{lem}\label{S6lem2}
{\sl Let $u=(u^\h, u^3)$ be a smooth solenoidal vector field on $\Om.$ Then one has
\beno
\int_0^t\mid\bigl(\D_k^\h(u\cdot\na u) | \D_k^\h u\bigr)_{L^2(\Om)}\mid\,dt'\lesssim
d_k^22^{-4k}\|u\|_{\wt{L}^\infty_t(\cB^2_\h)}\bigl(\|\p_3u\|_{L^2_t(L^2)}^2+\|\p_3 u\|_{\wt{L}^2_t(\cB^2_\h)}^2\bigr).
\eeno}
\end{lem}

\begin{proof} To estimate the trilinear term $\bigl(\D_k^\h(u\cdot\na u) | \D_k^\h u\bigr)_{L^2(\Om)},$
 we have to distinguish  the terms containing the horizontal derivatives with the term containing the vertical derivative. We write
\beq\label{S6eq2}
\begin{split}
\bigl(\D_k^\h(u\cdot\na u) &| \D_k^\h u\bigr)_{L^2}=I_\h+I_\v\with\\
I_\h\eqdefa \bigl(\Delta_k^\h(u^\h\cdot\nabla_\h u) | &\Delta_k u\bigr)_{L^2}\quad\text{and}\quad I_\v\eqdefa \bigl(\Delta_k^h(u^3\partial_3 u) | \Delta_k^\h u\bigr).
\end{split}
\eeq
We start with the estimate of $I_\h$. Applying Bony's decomposition \eqref{pd} to $u^\h\cdot\nabla_\h u$ in the horizontal variables gives
\beno
u^\h\cdot\nabla_\h u=T^\h_{u^\h}\nabla_\h u+\bar{R}^\h(u^\h,\nabla_\h u).
\eeno
Through a commutative argument, we find
\beno
\begin{split}
\int_0^t\bigl(\D_k^\h(T^\h_{u^\h}\nabla_\h u) | \D_k^\h u\bigr)_{L^2}\,dt'=&\sum_{|k'-k|\leq 4}\Bigl(\int_0^t\bigl([\D_k^\h; S_{k'-1}^\h u^\h]\cdot\na_\h\D_{k'}^\h u | \D_k^\h u\bigr)_{L^2}\,dt'\\
&\qquad+\int_0^t\bigl((S_{k'-1}^\h u^\h-S_{k-1}^\h u^\h)\cdot\na_\h\D_{k'}^\h\D_k^\h u | \D_k^\h u\bigr)_{L^2}\,dt'\Bigr)\\
&+\int_0^t\bigl(S_{k-1}^\h u^\h\cdot\na_\h\D_k^\h u | \D_k^\h u\bigr)_{L^2}\,dt'\\
\eqdefa & I_\h^1(t)+I_\h^2(t)+I_\h^3(t).
\end{split}
\eeno
It follows from standard commutator's estimate (see \cite{bcd}) and Lemma \ref{lemBern} that
\beno
\begin{split}
\|[\Delta_k^h; S_{k'-1}^\h u^\h]\nabla_\h \Delta_{k'}^\h u^\h\|_{L^2}\leq& C2^{-k}\|\nabla_\h S_{k'-1}^\h u^\h\|_{L^\infty}\|\Delta_{k'}^\h \nabla_\h u^h\|_{L^2}\\
\leq & C\|\nabla_\h u\|_{L^\infty}\|\Delta_{k'}^\h u\|_{L^2},
\end{split}
\eeno
so that we deduce from  Lemma \ref{S6lem1} that
\beno
\begin{split}
\mid I_\h^1(t)\mid \lesssim &\sum_{|k'-k|\leq 4}\|\na_\h u\|_{L^2_t(L^\infty)}\|\D_{k'}^\h u\|_{L^\infty_t(L^2)}\|\D_{k'}^\h \p_3u\|_{L^2_t(L^2)}\\
\lesssim &d_k^22^{-4k}\|u\|_{\wt{L}^\infty_t(\cB^2_\h)}\|\p_3 u\|_{\wt{L}^2_t(\cB^2_\h)}^2.
\end{split}
\eeno
The same estimate holds for $I^2_\h(t).$

To handle $I_\h^3(t),$ we  first get,  by using integrating by parts and $\dive u=0,$ that
\beno
\begin{split}
I^3_\h(t)=&-\f12\int_0^t \bigl(S_{k-1} \dive_\h u^\h \Delta_k^\h u | \Delta_k^\h u\bigr)_{L^2}\,dt'\\
=&\frac{1}{2}\int_0^t \bigl(S_{k-1} \p_3u^3 \Delta_k^\h u | \Delta_k^\h u\bigr)_{L^2}\,dt',
\end{split}
\eeno
from which, we infer
\beno
\begin{split}
\mid I^3_\h(t)\mid\lesssim & \int_0^t \|S_{k-1}^\h\partial_3 u^3\|_{L^2_\v(L^\infty_\h)}\|\Delta_k^\h u\|_{L^4_\v(L^2_\h)}^2\,dt'\\
\lesssim & \int_0^t\|\partial_3 u^3\|_{\cB^1_\h}\|\Delta_k^\h u\|_{L^2}^{\frac 32}\| \Delta_k^\h\partial_3 u\|_{L^2}^{\frac 12}\,dt'.
\end{split}
\eeno
Applying \eqref{S6eq1} gives
\beno
\begin{split}
\mid I^3_\h(t)\mid
\lesssim &\int_0^t \|\p_3u\|_{L^2}^{\f12}\|\partial_3 u^3\|_{\cB^2_\h}^{\f12}\|\Delta_k^\h u\|_{L^2}\| \Delta_k^\h\partial_3 u\|_{L^2}\,dt'\\
\lesssim & \|\p_3u\|_{L^2_t(L^2)}^{\f12}\|\partial_3 u^3\|_{{L}^2_t(\cB^2_\h)}^{\f12}\|\Delta_k^\h u\|_{L^\infty_t(L^2)}\| \Delta_k^\h\partial_3 u\|_{L^2_t(L^2)}\\
\lesssim & d_k^22^{-4k}\|u\|_{\wt{L}^\infty_t(\cB^2_\h)}\|\p_3u\|_{L^2_t(L^2)}^{\f12}\|\p_3 u\|_{\wt{L}^2_t(\cB^2_\h)}^{\f32}.
\end{split}
\eeno
As a result, it comes out
\beq \label{S6eq3}
\begin{split}
\mid\int_0^t\bigl(\D_k^\h (T^\h_{u^\h}\nabla_\h u) | \D_k^\h u\bigr)_{L^2}\,dt'\mid\lesssim  d_k^22^{-4k}\|u\|_{\wt{L}^\infty_t(\cB^2_\h)}\bigl(\|\p_3u\|_{L^2_t(L^2)}^2+\|\p_3 u\|_{\wt{L}^2_t(\cB^2_\h)}^2\bigr).
\end{split}
\eeq
Whereas observing that
\beno
\begin{split}
\mid\int_0^t\bigl(\D_k^\h (\bar{R}^\h({u^\h},\nabla_\h u)) | \D_k^\h u\bigr)_{L^2}\,dt'\mid\lesssim &\sum_{k'\geq k-N_0}\int_0^t\|\D_{k'}^\h
u^\h\|_{L^2}\|S_{k'+2}^\h \na_\h u\|_{L^\infty}\|\D_k^\h u\|_{L^2}\,dt'\\
\lesssim &\sum_{k'\geq k-N_0}\|\na_\h u\|_{L^2_t(L^\infty)}\|\D_{k'}^\h u^\h\|_{L^\infty_t(L^2)}\|\D_k^\h u\|_{L^2_t(L^2)},
\end{split}
\eeno
which together with \eqref{S6eq1} and Lemma \ref{S6lem1} ensures that
\beno
\begin{split}
\mid\int_0^t\bigl(\D_k^\h (\bar{R}^\h({u^\h},\nabla_\h u)) | \D_k^\h u\bigr)_{L^2}\,dt'\mid
\lesssim & d_k2^{-2k}\sum_{k'\geq k-N_0}d_{k'}2^{-2k'}\|u\|_{\wt{L}^\infty_t(\cB^2_\h)}\|\p_3 u\|_{\wt{L}^2_t(\cB^2_\h)}^2\\
\lesssim & d_k^22^{-4k}\|u\|_{\wt{L}^\infty_t(\cB^2_\h)}\|\p_3 u\|_{\wt{L}^2_t(\cB^2_\h)}^2.
\end{split}
\eeno
Along with \eqref{S6eq3}, we conclude that
\beq \label{S6eq4}
\mid I_\h\mid \lesssim  d_k^22^{-4k}\|u\|_{\wt{L}^\infty_t(\cB^2_\h)}\bigl(\|\p_3u\|_{L^2_t(L^2)}^2+\|\p_3 u\|_{\wt{L}^2_t(\cB^2_\h)}^{2}\bigr).
\eeq

To handle $I_\v(t),$ we get, by applying Bony's decomposition \eqref{pd} to $u^3\p_3u$ in the horizontal variables, that
\beno
u^3\p_3u=T^\h_{u^3}\p_3u+\bar{R}^\h(u^3,\p_3u).
\eeno
We first observe from Lemma \ref{S6lem1} that
\beno
\begin{split}
\int_0^t\mid\bigl(\D_k^\h (T^\h_{u^3}\p_3u) | \D_k^\h u\bigr)_{L^2}\mid \,dt'\lesssim
&\sum_{|k'-k|\leq 4}\int_0^t\|S_{k'-1}^\h u^3\|_{L^\infty}\|\D_{k'}^\h\p_3u\|_{L^2}\|\D_{k}^\h u\|_{L^2}\,dt'\\
\lesssim
&\sum_{|k'-k|\leq 4}\| \p_3u^3\|_{L^2_t(\cB^1_\h)}\|\D_{k'}^\h\p_3u\|_{L^2_t(L^2)}\|\D_{k}^\h u\|_{L^\infty_t(L^2)}\\
\lesssim & d_k^22^{-4k}\|u\|_{\wt{L}^\infty_t(\cB^2_\h)}\|\p_3u\|_{L^2_t(L^2)}^{\f12}\|\p_3 u\|_{\wt{L}^2_t(\cB^2_\h)}^{\f32}.
\end{split}
\eeno
Similarly, we get, by applying \eqref{S6eq1}, that
\beno
\begin{split}
\int_0^t\mid&\bigl(\D_k^\h (R^\h({u^3},\p_3u)) | \D_k^\h u\bigr)_{L^2}\mid \,dt'\\
\lesssim
&\sum_{k'\geq k-N_0}\int_0^t\|\D_{k'}^\h u^3\|_{L^\infty_\v(L^2)}\|S_{k'+2}^\h\p_3u\|_{L^2_\v(L^\infty_\h)}\|\D_{k}^\h u\|_{L^2}\,dt'\\
\lesssim
&\sum_{k'\geq k-N_0}\int_0^t\|\p_3u\|_{\cB^1_\h}\|\D_{k'}^\h u^3\|_{L^2}^{\f12}\|\D_{k'}^\h \p_3u^3\|_{L^2}^{\f12}\|\D_{k}^\h u\|_{L^2}\,dt'\\
\lesssim
&\sum_{k'\geq k-N_0}\| \p_3u^3\|_{L^2_t(\cB^1_\h)}\|\D_{k'}^\h\p_3u\|_{L^2_t(L^2)}\|\D_{k}^\h u\|_{L^\infty_t(L^2)}\\
\lesssim & d_k^22^{-4k}\|u\|_{\wt{L}^\infty_t(\cB^2_\h)}\|\p_3u\|_{L^2_t(L^2)}^{\f12}\|\p_3 u\|_{\wt{L}^2_t(\cB^2_\h)}^{\f32}.
\end{split}
\eeno
This leads to
\beno
\mid I_\v\mid \lesssim  d_k^22^{-4k}\|u\|_{\wt{L}^\infty_t(\cB^2_\h)}\|\p_3u\|_{L^2_t(L^2)}^{\f12}\|\p_3 u\|_{\wt{L}^2_t(\cB^2_\h)}^{\f32}.
\eeno
Together with \eqref{S6eq2} and \eqref{S6eq4}, we complete the proof of Lemma \ref{S6lem2}.
\end{proof}

Let now present the proof of Theorem \ref{thm3}.

\begin{proof}[Proof of Theorem \ref{thm3}] For simplicity, we only present the {\it a priori} estimates for smooth enough solutions of \eqref{S1eq4}.
By taking $L^2$ inner product of the momentum equation of \eqref{S1eq4} with $u,$ we get
$$\frac{1}{2}\frac{d}{dt}\|u(t)\|^2_{L^2}+\nu_\v\|\partial_3 u\|^2_{L^2}= 0.$$
Integrating the above inequality over $[0,t]$ leads to
\beno
\frac12\|u(t)\|_{L^2}^2+\nu_\v\|\partial_3 u\|_{L^2_t(L^2)}^2\leq \frac12\|u_0\|_{L^2}^2,
\eeno
which implies
\beq \label{S6eq5}
\|u\|_{L^\infty(L^2)}+\nu^{\f12}_\v\|\partial_3 u\|_{L^2_t(L^2)}\leq \sqrt{2}\|u_0\|_{L^2}.
\eeq

On the other hand, by
applying the operator $\Delta_k^\h $ to the momentum equation  of \eqref{S1eq4} and performing $L^2$ inner product of the resulting
equation with $\D_k^\h u,$  we obtain
$$\frac{1}{2}\frac{d}{dt}\|\Delta_k^\h u\|_{L^2}^2+\nu_\v\|\partial_3 \Delta_k^h u\|^2_{L^2}\leq
\mid\bigl(\Delta_k^\h (u\cdot\nabla u) | \Delta_k^\h u\bigr)_{L^2}\mid.$$
Integrating the above inequality over $[0,t]$ and then applying Lemma \ref{S6lem2} yields
\beno
\begin{split}
\|\D_k^\h u\|_{L^\infty_t(L^2)}^2+&\nu_\v\|\partial_3 \Delta_k^h u\|^2_{L^2_t(L^2)}\\
\leq& \|\D_k^\h u_0\|_{L^2}^2+C
d_k^22^{-4k}\|u\|_{\wt{L}^\infty_t(\cB^2_\h)}\bigl(\|\p_3u\|_{L^2_t(L^2)}^2+\|\p_3 u\|_{\wt{L}^2_t(\cB^2_\h)}^2\bigr).
\end{split}
\eeno
By taking square root of the above inequality and multiplying the resulting inequality by $2^{2k},$ and then summing up
the resulting inequality for $k\in\Z,$ we achieve
\beq\label{S6eq6}
\|u\|_{\wt{L}^\infty_t(\cB^2_\h)}+\nu_\v^{\f12}\|\p_3 u\|_{\wt{L}^2_t(\cB^2_\h)}
\leq C\Bigl(\|u_0\|_{\cB_\h^2}+\|u\|_{\wt{L}^\infty_t(\cB^2_\h)}^{\f12}\bigl(\|\p_3u\|_{L^2_t(L^2)}+\|\p_3 u\|_{\wt{L}^2_t(\cB^2_\h)}\bigr)\Bigr).
\eeq
Summing up \eqref{S6eq5} with \eqref{S6eq6} gives rise to
\beq \label{S6eq7}
\begin{split}
\|u\|_{L^\infty(L^2)}+&\|u\|_{\wt{L}^\infty_t(\cB^2_\h)}+\nu^{\f12}\bigl(\|\partial_3 u\|_{L^2_t(L^2)}+\|\p_3 u\|_{\wt{L}^2_t(\cB^2_\h)}\bigr)\\
\leq &C\Bigl(\|u_0\|_{L^2}+\|u_0\|_{\cB_\h^2}+\|u\|_{\wt{L}^\infty_t(\cB^2_\h)}^{\f12}\bigl(\|\p_3u\|_{L^2_t(L^2)}+\|\p_3 u\|_{\wt{L}^2_t(\cB^2_\h)}\bigr)\Bigr).
\end{split}
\eeq
Thanks to \eqref{S6eq7}, we deduce by a standard argument that \eqref{S1eq4} has a unique solution
$u\in C([0, T^\ast[; L^2\cap \cB_h^2(\Om))$ with $\p_3u\in L^2(]0, T^\ast[; L^2\cap \cB_h^2(\Om))$ for some maximal
existing time $T^\ast.$ We are going to prove that $T^\ast=\infty$ under the assumption of \eqref{S1eq5}.

Let us denote
\beq \label{S6eq8}
T^\star\eqdefa \sup\bigl\{ T<T^\ast, \ 4C^2\|u\|_{\wt{L}^\infty_t(\cB^2_\h)}\leq \nu_\v\ \bigr\}.
\eeq
Then for $t\leq T^\star,$ we deduce from \eqref{S6eq7} that
\beq \label{S6eq9}
\|u\|_{L^\infty(L^2)}+\|u\|_{\wt{L}^\infty_t(\cB^2_\h)}+\f{\nu^{\f12}}2\bigl(\|\partial_z u\|_{L^2_t(L^2)}+\|\p_3 u\|_{\wt{L}^2_t(\cB_\h^2)}\bigr)
\leq C\bigl(\|u_0\|_{L^2}+\|u_0\|_{\cB_\h^2}\bigr).
\eeq
In particular, under the assumption of \eqref{S1eq5},  for $t\leq T^\star,$ we have
\beno
\|u\|_{L^\infty(L^2)}+\|u\|_{\wt{L}^\infty_t(\cB^2_\h)}+\f{\nu^{\f12}}2\bigl(\|\partial_z u\|_{L^2_t(L^2)}+\|\p_3 u\|_{\wt{L}^2_t(\cB_\h^2)}\bigr)
\leq \f{\nu_\v}{8C^2}
\eeno as long as $c$ in \eqref{S1eq5} is small enough.
This contradicts with the definition of $T^\star$ determined by \eqref{S6eq8}, which  in turn shows that $T^\star=T^\ast=\infty.$ We complete
the proof of Theorem \ref{thm3}.
\end{proof}

\appendix

\setcounter{equation}{0}
\section{The proof of  Lemmas \ref{S4lem2} to \ref{S4lem6}}\label{appenda}

The goal of this section is to  present the proof of Lemmas \ref{S4lem2} to \ref{S4lem6}.

\begin{proof}[Proof of Lemma \ref{S4lem2}] Applying Bony's decomposition \eqref{pd} to $a\otimes [\buh]_\e$ in the vertical variable yields
\beno
a\otimes [\buh]_\e=T^\v_a[\buh]_\e+\bar{R}^\v(a,[\buh]_\e).
\eeno
Considering the support properties to the Fourier transform of the terms in $T^\v_a[\buh]_\e,$ we find
\beno
\begin{split}
\int_0^t&\bigl|\bigl(\D_\ell^\v(T^\v_{a}[\buh]_\e) | \D_\ell^\v b\bigr)_{L^2}\bigr|\,dt'\\
\leq&\sum_{|\ell'-\ell|\leq 4}\int_0^t\|S_{\ell'-1}^{\v} a\|_{L^\infty_\v(L^4_\h)}\|\D_{\ell'}^\v[\buh]_\e\|_{L^2_\v(L^4_\h)}\|\D_\ell^\v b\|_{L^2}\,dt'\\
\lesssim &\sum_{|\ell'-\ell|\leq 4}2^{-\f{\ell'}2}\int_0^t d_{\ell'}(t')\|a\|_{L^\infty_\v(L^4_\h)}\|[\buh]_\e\|_{B^{0,\f12}}^{\f12}\|[\na_\h\buh]_\e\|_{B^{0,\f12}}^{\f12}\|\D_\ell^\v b\|_{L^2}\,dt'\\
\lesssim &\sum_{|\ell'-\ell|\leq 4}d_{\ell'}2^{-\f{\ell'}2}\int_0^t \|a\|_{L^\infty_\v(L^4_\h)}\|\buh\|_{B^{0,\f12}}^{\f12}\|\na_\h\buh\|_{B^{0,\f12}}^{\f12}
\|\D_\ell^\v b\|_{L^2}\,dt'.
\end{split}
\eeno
On the other hand,  it follows from Lemma \ref{lemBern} that
\beno
\begin{split}
\|\D_\ell^\v a\|_{L^\infty_\v(L^4_\h)}\lesssim& 2^{\f{\ell}2}\|\D_\ell^\v a\|_{L^2_\v(L^4_\h)}\\
\lesssim &2^{\f{\ell}2}\|\D_\ell^\v a\|_{L^2}^{\f12}\|\D_\ell^\v \na_\h a\|_{L^2}^{\f12},
\end{split}
\eeno
from which and Definition \ref{defpz}, we infer
\beno
\begin{split}
\Bigl(\int_0^t& \|\buh\|_{B^{0,\f12}}\|\na_\h\buh\|_{B^{0,\f12}}\|a\|_{L^\infty_\v(L^4_\h)}^2\,dt'\Bigr)^{\f12}\\
\leq &\sum_{\ell\in\Z} 2^{\f{\ell}2} \Bigl(\int_0^t \|\buh\|_{B^{0,\f12}}\|\na_\h\buh\|_{B^{0,\f12}}\|\D_\ell^\v a\|_{L^2}\|\D_\ell^\v \na_\h a\|_{L^2}\,dt'\Bigr)^{\f12}\\
\lesssim &\sum_{\ell\in\Z}2^{\f{\ell}2}\Bigl(\int_0^t f_1(t')\|\D_\ell^\v a\|_{L^2}^2\,dt'\Bigr)^{\f12}\|\D_\ell^\v\na_\h a\|_{L^2_t(L^2)}^{\f12}\\
\lesssim &\|a\|_{\wt{L}^2_{t,f_1}(B^{0,\f12})}^{\f12}\|\na_\h a\|_{\wt{L}^2_{t}(B^{0,\f12})}^{\f12},
\end{split}
\eeno for $f_1(t)$ being given by \eqref{S4eq1}.
As a result, it comes out
\beq \label{S5eq5}
\int_0^t\bigl|\bigl(\D_\ell^\v (T^\v_{a}[\buh]_\e) | \D_\ell^\v b\bigr)_{L^2}\bigr|\,dt\lesssim d_\ell^2
2^{-\ell}\|a\|_{\wt{L}^2_{t,f_1}(\cB^{0,\f12})}^{\f12}\|\na_\h a\|_{\wt{L}^2_{t}(\cB^{0,\f12})}^{\f12}\|b\|_{\wt{L}^2_{t}(\cB^{0,\f12})}.
\eeq

Exactly along the same line, we have
\beno
\begin{split}
\int_0^t&\bigl|\bigl(\D_\ell^\v (\bar{R}^\v(a,[\buh]_\e)) | \D_\ell^\v b\bigr)_{L^2}\bigr|\,dt'\\
\lesssim&\sum_{\ell'\geq \ell-N_0}\int_0^t\|\D_{\ell'}^\v a\|_{L^2_\v(L^4_\h)}\|S_{\ell'+2}^\v([\buh]_\e)\|_{L^\infty_\v(L^4_\h)}\|\D_\ell^\v b\|_{L^2}\,dt'\\
\lesssim&\sum_{\ell'\geq \ell-N_0}\int_0^t\|\buh\|_{L^\infty_\v(L^4_\h)}\|\D_{\ell'}^\v a\|_{L^2}^{\f12}\|\D_{\ell'}^\v \na_\h a\|_{L^2}^{\f12}\|\D_\ell^\v b\|_{L^2}\,dt'\\
\lesssim&\sum_{\ell'\geq \ell-N_0}\Bigl(\int_0^t f_1(t')\|\D_{\ell'}^\v a\|_{L^2}^2\,dt'\Bigr)^{\f14}\|\D_{\ell'}^\v \na_\h a\|_{L^2_t(L^2)}^{\f12}\|\D_\ell^\v b\|_{L^2_t(L^2)}\\
\lesssim & d_\ell
2^{-\ell}\|a\|_{\wt{L}^2_{t,f_1}(B^{0,\f12})}^{\f12}\|\na_\h a\|_{\wt{L}^2_{t}(B^{0,\f12})}^{\f12}\|b\|_{\wt{L}^2_{t}(B^{0,\f12})}\sum_{\ell'\geq \ell-N_0}d_{\ell'}2^{-\f{\ell'-\ell}2},
\end{split}
\eeno
which implies
\beno \int_0^t\bigl|\bigl(\D_\ell^\v(\bar{R}^\v(a,[\buh]_\e)) | \D_\ell^\v b\bigr)_{L^2}\bigr|\,dt'\lesssim d_\ell^2
2^{-\ell}\|a\|_{\wt{L}^2_{t,f_1}(B^{0,\f12})}^{\f12}\|\na_\h a\|_{\wt{L}^2_{t}(B^{0,\f12})}^{\f12}\|b\|_{\wt{L}^2_{t}(B^{0,\f12})}.
\eeno
Along with \eqref{S5eq5}, we complete the proof of Lemma \ref{S4lem2}.
\end{proof}

\begin{proof}[Proof of Lemma \ref{S4lem3}]
Applying Bony's decomposition \eqref{pd} to $[\buh]_\e\cdot\na_\h v_L$ in the vertical variable gives
\beno
[\buh]_\e\cdot\na_\h v_L=T^\v_{[\buh]_\e}\na_\h v_L+\bar{R}^\v([\buh]_\e,\na_\h v_L).
\eeno
We first observe that
\beno
\begin{split}
\int_0^t&\bigl|\bigl(\D_\ell^\v (T^\v_{[\buh]_\e}\na_\h v_L) | \D_\ell^\v  a\bigr)_{L^2}\bigr|\,dt'\\
\lesssim&\sum_{|\ell'-\ell|\leq 4}\int_0^t\|S_{\ell'-1}^\v([\buh]_\e)\|_{L^\infty_\v(L^4_\h)}\|\D_{\ell'}^\v \na_\h v_L\|_{L^2}\|\D_{\ell}^\v a\|_{L^2_\v(L^4_\h)}\,dt'\\
\lesssim&\sum_{|\ell'-\ell|\leq 4}\int_0^t\|\buh\|_{L^\infty_\v(L^4_\h)}\|\D_{\ell'}^\v \na_\h v_L\|_{L^2}\|\D_{\ell}^\v a\|_{L^2}^{\f12}\|\D_{\ell}^\v \na_\h a\|_{L^2}^{\f12}\,dt'\\
\lesssim&\sum_{|\ell'-\ell|\leq 4}\|\D_{\ell'}^\v \na_\h v_L\|_{L^2_t(L^2)}\Bigl(\int_0^t f_1(t)\|\D_{\ell}^\v a\|_{L^2}^2\,dt\Bigr)^{\f14}\|\D_{\ell}^\v \na_\h a\|_{L^2_t(L^2)}^{\f12}\\
\lesssim & d_\ell^2
2^{-\ell}\|\na_\h v_L\|_{\wt{L}^2_{t}(B^{0,\f12})}\|a\|_{\wt{L}^2_{t,f_1}(B^{0,\f12})}^{\f12}\|\na_\h a\|_{\wt{L}^2_{t}(\cB^{0,\f12})}^{\f12},
\end{split}
\eeno where we used Definition \ref{defpz} in the last step.

Similarly, we have
\beno
\begin{split}
\int_0^t&\bigl|\bigl(\D_\ell^\v(\bar{R}^\v([\buh]_\e, {\na_\h v_L})) | \D_\ell^\v  a\bigr)_{L^2}\bigr|\,dt'\\
\lesssim&\sum_{\ell'\geq \ell-N_0}\int_0^t\|\D_{\ell'}^\v ([\buh]_\e)\|_{L^2_\v(L^4_\h)}\|S_{\ell'+2}^\v \na_\h v_L\|_{L^\infty_\v(L^2_\h)}\|\D_{\ell}^\v a\|_{L^2_\v(L^4_\h)}\,dt'\\
\lesssim&\sum_{\ell'\geq \ell-N_0}2^{-\f{\ell'}2}\int_0^Td_{\ell'}(t)\|\na_\h v_L\|_{B^{0,\f12}}\|\buh\|_{B^{0,\f12}}^{\f12}\|\na_\h\buh\|_{B^{0,\f12}}^{\f12}
\|\D_{\ell}^\v a\|_{L^2}^{\f12}\|\D_{\ell}^\v \na_\h a\|_{L^2}^{\f12}\,dt',
\end{split}
\eeno then by applying H\"older's inequality and using Definition \ref{defpz}, we obtain
\beno
\begin{split}
\int_0^t&\bigl|\bigl(\D_\ell^\v(\bar{R}^\v([\buh]_\e, {\na_\h v_L})) | \D_\ell^\v  a\bigr)_{L^2}\bigr|\,dt'\\
\lesssim&\sum_{\ell'\geq \ell-N_0}d_{\ell'}2^{-\f{\ell'}2}\|\na_\h v_L\|_{\wt{L}^2_t(B^{0,\f12})} \Bigl(\int_0^t f_1(t')\|\D_{\ell}^\v a\|_{L^2}^2\,dt'\Bigr)^{\f14}\|\D_{\ell}^\v \na_\h a\|_{L^2_t(L^2)}^{\f12}\\
\lesssim & d_\ell^2
2^{-\ell}\|\na_\h v_L\|_{\wt{L}^2_{t}(B^{0,\f12})}\|a\|_{\wt{L}^2_{t,f_1}(B^{0,\f12})}^{\f12}\|\na_\h a\|_{\wt{L}^2_{t}(B^{0,\f12})}^{\f12}.
\end{split}
\eeno
This completes the proof of Lemma \ref{S4lem3}.
\end{proof}

\begin{proof}[Proof of Lemma \ref{S4lem5}] By applying Bony's decomposition \eqref{pd} to $v_L\cdot[\na_\e\buh]_\e$ in the vertical variable to $v_L\cdot[\na_\e\buh]_\e,$ we write
\beno
v_L\cdot[\na_\e\buh]_\e=T^\v_{v_L}[\na_\e\buh]_\e+\bar{R}^\v(v_L,[\na_\e\buh]_\e).
\eeno
Note that
\beno
\begin{split}
&\int_0^te^{-\la_2\int_0^{t'}f_2(\tau)\,d\tau}\bigl|\bigl(\D_\ell^\v (T^\v_{v_L}[\na_\e \buh]_\e) | \D_\ell^\v  a\bigr)_{L^2}\bigr|\,dt'\\
&\lesssim \sum_{|\ell'-\ell|\leq 4}\int_0^te^{-\la_2\int_0^{t'}f_2(\tau)\,d\tau}\|S_{\ell'-1}^\v v_L\|_{L^\infty_\v(L^4_\h)}\|\D_{\ell'}^\v [\na_\e \buh]_\e\|_{L^2}
\|\D_\ell^\v a\|_{L^2_\v(L^4_\h)}\,dt'\\
&\lesssim \sum_{|\ell'-\ell|\leq 4}2^{-\f{\ell'}2}\int_0^t d_{\ell'}(t') e^{-\la_2\int_0^{t'}f_2(\tau)\,d\tau}\|v_L\|_{L^\infty_\v(L^4_\h)}\|\na_\e \buh\|_{\cB^{0,\f12}}
\|\D_\ell^\v a\|_{L^2}^{\f12}\|\D_\ell^\v \na_\h a\|_{L^2}^{\f12}\,dt'.
\end{split}
\eeno
Applying H\"older's inequality gives
\beq \label{S4eq15}
\begin{split}
&\int_0^te^{-\la_2\int_0^{t'}f_2(\tau)\,d\tau}\bigl|\bigl(\D_\ell^\v (T^\v_{v_L}[\na_\e \buh]_\e) | \D_\ell^\v  a\bigr)_{L^2}\bigr|\,dt'\\
&\lesssim \sum_{|\ell'-\ell|\leq 4} d_{\ell'}2^{-\f{\ell'}2}\|v_L\|_{\wt{L}^\infty_t(B^{0,\f12})}^{\f12}
 \|\na_\h v_L\|_{\wt{L}^2_t(B^{0,\f12})}^{\f12}\|\na_\h \D_{\ell}^\v a\|_{L^2_t(\cB^{0,\f12})}^{\f12}\\
 &\qquad\qquad\qquad\times \Bigl(\int_0^t f_2(t')e^{-4\la_2\int_0^{t'}f_2(\tau)\,d\tau}\,dt'\Bigr)^{\f14}
 \Bigl(\int_0^t f_2(t')\|\D_\ell^\v a\|_{L^2}^2\,dt'\Bigr)^{\f14}\\
&\lesssim  d_\ell^2 2^{-{\ell}}\la_2^{-\f14} \|v_L\|_{\wt{L}^\infty_t(\cB^{0,\f12})}^{\f12}
\|\na_\h v_L\|_{\wt{L}^2_t(\cB^{0,\f12})}^{\f12}\| a\|_{\wt{L}^2_{t,f_2}(B^{0,\f12})}^{\f12}
\| \na_\h a\|_{\wt{L}^2_{t}(B^{0,\f12})}^{\f12}.
 \end{split}
 \eeq

 On the other hand, we observe that
 \beno
\begin{split}
&\int_0^te^{-\la_2\int_0^{t'}f_2(\tau)\,d\tau}\bigl|\bigl(\D_\ell^\v (\bar{R}^\v(v_L,[\na_\e \buh]_\e)) | \D_\ell^\v  a\bigr)_{L^2}\bigr|\,dt'\\
&\lesssim \sum_{\ell'\geq \ell-N_0}\int_0^te^{-\la_2\int_0^{t'}f_2(\tau)\,d\tau}\|\D_{\ell'}^\v v_L\|_{L^2_\v(L^4_\h)}\|S_{\ell'+2}^\v [\na_\e \buh]_\e\|_{L^\infty_\v(L^2_\h)}
\|\D_\ell^\v a\|_{L^2_\v(L^4_\h)}\,dt'\\
&\lesssim \sum_{\ell'\geq \ell-N_0}\int_0^t  e^{-\la_2\int_0^{t'}f_2(\tau)\,d\tau}\|\D_{\ell'}^v v_L\|_{L^2}^{\f12}\|\D_{\ell'}^v \na_\h v_L\|_{L^2}^{\f12}
\|\na_\e \buh\|_{\cB^{0,\f12}}
\|\D_\ell^\v a\|_{L^2}^{\f12}\|\D_\ell^\v \na_\h a\|_{L^2}^{\f12}\,dt'.
\end{split}
\eeno
Then along the same line to proof of \eqref{S4eq15}, we arrive at
 \beno
\begin{split}
&\int_0^te^{-\la_2\int_0^{t'}f_2(\tau)\,d\tau}\bigl|\bigl(\D_\ell^\v \bar{R}^\v(v_L,[\na_\e \buh]_\e)) | \D_\ell^\v  a\bigr)_{L^2}\bigr|\,dt'\\
&\lesssim   d_\ell^2 2^{-{\ell}}\la_2^{-\f14}\|v_L\|_{\wt{L}^\infty_t(B^{0,\f12})}^{\f12}
\|\na_\h v_L\|_{\wt{L}^2_t(B^{0,\f12})}^{\f12}\| a\|_{\wt{L}^2_{t,f_2}(B^{0,\f12})}^{\f12}
\| \na_\h a\|_{\wt{L}^2_{t}(\cB^{0,\f12})}^{\f12}.
 \end{split}
 \eeno
This together with \eqref{S4eq15} completes the proof of Lemma \ref{S4lem5}.
\end{proof}

\begin{proof}[Proof of Lemma \ref{S4lem6}]
In view of \eqref{S2eq5},
we get, by applying Bony's decomposition  \eqref{pd} to $[\p_z\bar{u}^i]_\e[\bar{u}^j]_\e$ in the vertical variable that
\beq \label{S4eq8}
\begin{split}
\p_3[p^\h]_\e=&2\e\sum_{i,j=1}^2(-\D_\h)^{-1}\p_i\p_j\bigl([\p_z\bar{u}^i]_\e[\bar{u}^j]_\e\bigr)\\
=&2\e\sum_{i,j=1}^2(-\D_\h)^{-1}\p_i\p_j\bigl(T^\v_{[\p_z\bar{u}^i]_\e}[\bar{u}^j]_\e+\bar{R}^\v([\p_z\bar{u}^i]_\e, [\bar{u}^j]_\e)\bigr).
\end{split}
\eeq
We first observe from  Lemma \ref{lemBern}  that for any $\th\in [0,1/2[$
\beno
\begin{split}
\int_0^t&\exp\Bigl(-\la_3\int_0^{t'}f_3(\tau)\,d\tau\Bigr)\bigl|\bigl((-\D_\h)^{-1}\p_i\p_j\D_\ell^\v (T^\v_{[\p_z\bar{u}^i]_\e}[\bar{u}^j]_\e) | \D_\ell^\v a\bigr)_{L^2}\bigr|\,dt'\\
\lesssim & \sum_{|\ell'-\ell|\leq 4}\int_0^t\exp\Bigl(-\la_3\int_0^{t'}f_3(\tau)\,d\tau\Bigr)\|S_{\ell'-1}^\v([\p_z\bar{u}^i]_\e)\|_{L^\infty_\v(L^4_\h)}
\|\D_{\ell'}^\v([\bar{u}^j]_\e)\|_{L^2_\v(L^4_\h)}\|\D_\ell^\v a\|_{L^2}\,dt'\\
\lesssim & 2^{-\ell\th} \sum_{|\ell'-\ell|\leq 4}\int_0^t\exp\Bigl(-\la_3\int_0^{t'}f_3(\tau)\,d\tau\Bigr)\|\p_z\buh\|_{L^\infty_\v(\dH^{\f12}_\h)}
 \|\D_{\ell'}^\v([|D_\h|^{\f12}\bar{u}^j]_\e)\|_{L^2}\\
 &\qquad\qquad\qquad\qquad\qquad\qquad\qquad\qquad\qquad\qquad\qquad\qquad\times\|\D_{\ell}^\v\p_3 a\|_{L^2}^\th \|\D_\ell^\v a\|_{L^2}^{1-\th}\,dt',
 \end{split}
\eeno
where $|D_\h|^{\f12}$ denotes Fourier multiplier with symbol $|\xi_\h|^{\f12}.$
Then  noticing the definition of $f_3(t)$ given by \eqref{S4eq1}, we get, by applying  H\"older's inequality  and Definition \ref{defpz}, that
 \beno
\begin{split}
\int_0^t&\exp\Bigl(-\la_3\int_0^{t'}f_3(\tau)\,d\tau\Bigr)\bigl|\bigl((-\D_\h)^{-1}\p_i\p_j\D_\ell^\v (T^\v_{[\p_z\bar{u}^i]_\e}[\bar{u}^j]_\e) | \D_\ell^\v a\bigr)_{L^2}\bigr|\,dt'\\
\lesssim& \e^{-\th}2^{-\ell\th}\sum_{|\ell'-\ell|\leq 4}2^{-\ell'\left(\f{1}2-\th\right)}\int_0^t d_{\ell'}(t')\exp\Bigl(-\la_3\int_0^{t'}f_3(\tau)\,d\tau\Bigr)\|\p_z\buh\|_{\cB^{\f12,\f12}}\\
 &\qquad\qquad\qquad\qquad\qquad\qquad\qquad \times\|\buh\|_{\cB^{\f12,\f12-\th}}
\|\D_\ell^\v\p_3 a\|_{L^2}^\th\|\D_\ell^\v a\|_{L^2}^{1-\th}\,dt'\\
\lesssim & \e^{-\th}2^{-\ell\th}\sum_{|\ell'-\ell|\leq 4}d_{\ell'} 2^{-\ell'\left(\f{1}2-\th\right)}
\Bigl(\int_0^{t}f_3(t')\exp\Bigl(-2\la_3\int_0^{t'}f_3(\tau)\,d\tau\Bigr)\,dt'\Bigr)^{\f12}\\
&\qquad\times \Bigl(\int_0^t \|\buh\|_{\cB^{\f12,\f12-\th}}^{\f2{1-\th}}\|\D_\ell^\v a\|_{L^2}^2\,dt'\Bigr)^{\f{1-\th}2}\|\D_\ell^\v\p_3 a\|_{L^2_t(L^2)}^\th\\
\lesssim& \e^{-\th} d_\ell^22^{-\ell}\la_3^{-\f12}\|a\|_{\wt{L}^2_{t,f_3}(B^{0,\f12})}^{1-\th}\|\p_3 a\|_{\wt{L}^2_t(B^{0,\f12})}^\th.
\end{split}
\eeno

Along the same line, we have
\beno
\begin{split}
\int_0^t&\exp\Bigl(-\la_3\int_0^{t'}f_3(\tau)\,d\tau\Bigr)\bigl|\bigl((-\D_\h)^{-1}\p_i\p_j\D_\ell^\v(\bar{R}^\v([\p_z\bar{u}^i]_\e,[\bar{u}^j]_\e)) | \D_\ell^\v a\bigr)_{L^2}\bigr|\,dt'\\
\lesssim & \sum_{\ell'\geq \ell-N_0}\int_0^t\exp\Bigl(-\la_3\int_0^{t'}f_3(\tau)\,d\tau\Bigr)\|{\D}_{\ell'}^\v([\p_z\bar{u}^i]_\e)\|_{L^2_\v(L^4_\h)}
\|S_{\ell'+2}^\v([\bar{u}^j]_\e)\|_{L^\infty_\v(L^4_\h)}\|\D_\ell^\v a\|_{L^2}\,dt'\\
\lesssim & \e^{-\th}2^{-\ell\th} \sum_{\ell'\geq\ell-N_0}d_{\ell'}2^{-\ell'\left(\f{1}2-\th\right)}\int_0^t\exp\Bigl(-\la_3\int_0^{t'}f_3(\tau)\,d\tau\Bigr)
\|\p_z\buh\|_{\cB^{\f12,\f12}}\\
&\qquad\qquad\qquad\qquad\qquad\qquad\qquad\times \|\buh\|_{\cB^{\f12,\f12-\th}}\|\D_\ell^\v\p_3a\|_{L^2}^\th\|\D_\ell^\v a\|_{L^2}^{1-\th}\,dt',
\end{split}
\eeno
which together with the fact that $\th\in [0,1/2[$ implies that
\beno
\begin{split}
\int_0^t&\exp\Bigl(-\la_3\int_0^{t'}f_3(\tau)\,d\tau\Bigr)\bigl|\bigl((-\D_\h)^{-1}\p_i\p_j\D_\ell^\v(\bar{R}^\v([\p_z\bar{u}^i]_\e,[\bar{u}^j]_\e)) | \D_\ell^\v a\bigr)_{L^2}\bigr|\,dt'\\
\lesssim& d_\ell \e^{-\th}2^{-\ell\left(\th+\f12\right)}\|a\|_{\wt{L}^2_{t,f_3}(B^{0,\f12})}^{1-\th}\|\p_3 a\|_{\wt{L}^2_t(B^{0,\f12})}^\th \sum_{\ell'\geq\ell-N_0}d_{\ell'}2^{-\ell'\left(\f{1}2-\th\right)}\\
\lesssim& \e^{-\th} d_\ell^22^{-\ell}\la_3^{-\f12}\|a\|_{\wt{L}^2_{t,f_3}(B^{0,\f12})}^{1-\th}\|\p_3 a\|_{\wt{L}^2_t(B^{0,\f12})}^\th.
\end{split}
\eeno
Then by virtue of \eqref{S4eq8}, we complete the proof of Lemma \ref{S4lem6}.
\end{proof}

\bigbreak \noindent {\bf Acknowledgments.} M. Paicu was partially supported by the Agence Nationale de la
Recherche, Project IFSMACS, Grant ANR-15-CE40-0010.
P. Zhang is partially supported
by NSF of China under Grants   11371347 and 11688101, Morningside Center of Mathematics of The Chinese Academy of Sciences and innovation grant from National Center for
Mathematics and Interdisciplinary Sciences.

\end{document}